\documentclass[11pt, a4paper]{article}
\usepackage{amsmath, amssymb, amsfonts, amsthm}
\usepackage{color}
\usepackage{bm}
\usepackage{subfigure}
\usepackage{graphicx}
\usepackage{multirow}
\usepackage{setspace}
\usepackage{algorithmic, epsf}
\usepackage{graphicx, color}
\usepackage{paralist}
\usepackage{float}
\usepackage{algorithm}
\usepackage{lineno}
\usepackage{verbatim}
\usepackage{blindtext}

\newcommand{\Rspace}        {\mathbb {R}}
\newcommand{\Zring}        {\mathbb {Z}}
         
\newcommand{\inter}        {\mathcal{X}}

\newtheorem{lemma}{Lemma}

\newtheorem{Theorem}{Theorem}
\newtheorem{theorem}{Theorem}
\newtheorem{proposition}{Proposition}
\newtheorem{corollary}{Corollary}
\newtheorem{Conjecture}{Conjecture}

\begin{document}
\title{On the Smith classes, the van Kampen obstruction and embeddability of $[3]*K$\thanks{This research was supported by a grant from IPM.}}

\author{Salman Parsa\thanks{ Current Address: Computer Science Department, Saint Louis University, Saint Louis, Missouri, USA.}}


\maketitle

\begin{abstract}
    In this survey-research paper, we first introduce the theory of Smith classes of complexes with fixed-point free, periodic maps on them. These classes, when defined for the deleted product of a simplicial complex $K$, are the same as the embedding classes of $K$. Embedding classes, in turn, are generalizations of the van Kampen obstruction class for embeddability of a $d$-dimensional complex $K$ into the Euclidean $2d$-space. All of these concepts will be introduced in simple terms. 
    
    Second, we use the theory
    introduced in the first part to relate the embedding classes (or the special Smith classes) of the the complex $[3]*K$ with the embedding classes of $K$. Here $[3]*K$ is the join of $K$ with a set of three points.
    
    Specifically, we prove that if the $m$-th embedding class of $K$ is non-zero, then the $(m+2)$-nd embedding class of $[3]*K$ is non-zero.
    We also prove some of the consequences of this theorem for the embeddability of $[3]*K$.

\end{abstract}

\section{Introduction}
Let $K$ be a $d$-dimensional simplicial complex. To $K$ is assigned an underlying topological space that we also denote by $K$. We say that $K$ \textit{embeds} into $\Rspace^n$ if there exists a one-to-one, piece-wise linear (PL) continuous function (map) $f:K \rightarrow \Rspace^n$. In other words, in this paper we restrict ourselves to the PL maps other than when explicitly stated otherwise. The following result is a special case of a well-known theorem (see Theorem \ref{t:grunbaum} below) proved by Gr\"unbaum \cite{Gru69}, continuing the work of Flores \cite{Flo33,Flo32}, van Kampen \cite{vKam33} and Rosen \cite{Ros60}. 

\begin{theorem}\label{t:triplelink}
Let $K$ be a graph that is not planar. Then, the join of $K$ with three vertices, $[3]*K$, does not embed into Euclidean 4-space $\Rspace^4$.  
\end{theorem}
From a different viewpoint, Theorem \ref{t:triplelink} says that, the intersection of links of three vertices in a 2-complex that is embeddable into $4$-space, is a planar graph. 

The proof of Theorem \ref{t:triplelink}, as given in \cite{Flo33,Ros60,Gru69}, is geometric and uses the Borsuk-Ulam theorem. In this paper, we generalize Theorem \ref{t:triplelink} to all dimensions ($d>2$) by algebraic methods using the definition of obstruction classes to existence of an embedding as special Smith classes. The method we exploit is more systematic than geometric methods and we obtain indeed stronger results. A major reason for writing this paper has been to advertise this systematic approach and the theory of Smith classes as an alternative to the Borsuk-Ulam theoerm. It is concievable that an algebraic method has a wider application.

We note that if in Theorem \ref{t:triplelink} instead of the term ``planar" we use the expression ``linklessly embeddable into 3-space", then, the theorem generalizes verbatim to all dimensions. That is, the intersection of three links of vertices in an embeddable complex into $\Rspace^{2d}$ has to be linklessly embeddable into $\Rspace^{2d-1}$. See \cite{Par18} for a proof and definition of ``linklessly embeddable".

\paragraph{}
Let $\Delta(K)$ denote the deleted product of $K$. The main result of this paper is the following.

\begin{theorem}\label{t:smithclass}
Let $K$ be a simplicial $d$-complex and let $A^m=A^{m}(\Delta(K),t)$, $m\geq 0$ even\footnote{For the purposes of this servey-research article we limit ourselves to this case.}, denote the $m$-th special Smith class of the deleted product of $K$ (= $m$-th embedding obstruction of $K$), where $t$ is the involution that exchanges the factors. If $A^m \neq 0$ Then $A^{m+2}(\Delta([3]*K))\neq 0$.      
\end{theorem}

Let $\vartheta(K)$ denote the van Kampen obstruction class of $K$. Theorem \ref{t:smithclass} is the main ingredient for our proof of the following.

\begin{theorem}\label{t:main}
Let $K$ be a $d$-dimensional simplicial complex. Then, $\vartheta(K) \neq 0$ if and only if $\vartheta([3]*K)\neq 0.$  
\end{theorem}

Since for $d>2$ the vanishing of the van Kampen obstruction implies that $K$ embeds into $\Rspace^{2d}$ the following corollary is immediate. 
\begin{corollary}\label{c:main}
Let $d\neq 2$ and $K$ be a $d$-dimensional simplicial complex that does not embed into $\Rspace^{2d}.$  Then $[3]*K$ does not embed into $\Rspace^{2d+2}$.  
\end{corollary}

From a different viewpoint, we have the following.
\begin{corollary}
Let $d>1$ be such that $d-1 \neq 2$ and $K$ a $d$-dimensional simplicial complex embeddable into $\Rspace^{2d}.$  Then the intersection of three links of vertices of $K$ is a $(d-1)$-complex that embeds into $\Rspace^{2d-2}$.  
\end{corollary}

\paragraph{Remark}
The results of this paper about embeddability of $[3]*K$ and its van Kampen obstruction can be proved using simpler geometric arguments. These proofs are presented in \cite{Par19} and use results of \cite{Sko02}. This approach works by assuming that a $\Zring_2$-equivariant map $\Delta([3]*K) \rightarrow S^{m+1}$ exists and then deduce the existence of an equivariant map $\Delta(K) \rightarrow S^{m-1}$ and vice versa, for appropriate $m$. However, in the case where the vanishing of embedding classes $A^{m+2}(\Delta([3]*K))$ does not imply existence of an equivariant map into $S^{m+1}$ the geometric approach seems to fail to prove Theorem~\ref{t:smithclass}.

\paragraph{}
We have also the following geometric characterization of the vanishing of the van Kampen obstruction.
\begin{corollary}
Let $K$ be a $d$-dimensional simplicial complex, $d \geq 1$. Then $\vartheta(K)=0$ if and only if $[3]*K$ embeds into $\Rspace^{2d+2}$.
\end{corollary}
If $\vartheta(K)=0$ then from Theorem \ref{t:main}, $\vartheta([3]*K)=0$. If $d > 1$ then $[3]*K$ embeds into $\Rspace^{2d+2}$ since the obstruction is complete in these dimensions. The case $d=1$ follows since the obstruction is complete for $d=1$ and hence $K$ is planar, and $[3]*K$ then embeds into $\Rspace^4$. Conversely, if $[3]*K$ embeds into $\Rspace^{2d+2}$ then $\vartheta([3]*K)=0$ and Theorem \ref{t:main} implies $\vartheta(K)=0$.

We also remark that, to the best of our knowledge, Theorem \ref{t:main} is the first result that constructs non-embeddable $(d+1)$-complexes (into $\Rspace^{2d+2}$), from arbitrary $d$-complexes, which have non-zero van Kampen obstruction. Using it, one can construct many new non-embeddable complexes by taking for $K$, for instance, infinitely many (minimal) non-embeddable complexes of Zaks \cite{Zak69} and Ummel \cite{Umm73}.

One goal of this paper is to provide a self-contained, concise, and perhaps inviting, treatment of the Smith classes as obstructions to embeddability, as defined by Wu \cite{Wu74}. We endeavor to keep the exposition simple and ``intuitive". The van Kampen obstruction can be defined as an special Smith class \cite{Wu74}. It has other definitions that we also explain. Moreover, we review and prove certain properties of the van Kampen obstruction. We also discuss the algorithmic and computational aspects of the various obstruction classes.

We hope that this exposition helps bringing to light the theory of Smith classes and that this theory is used in the future for attacking other geometric problems, as it has been the case for the embeddability problem.

\paragraph{The algebraic non-embeddability conjectures}
This study can be considered the first steps towards algebraic generalizations of the classical geometric non-embeddability results. These algebraic generalizations can be phrased as the following conjectures.
\begin{Conjecture}[Algebraic Gr\"unbaum-van Kampen-Flores Conjecture]
Let $K,L$ be simplicial complexes such that $\vartheta(K)\neq 0$ and $\vartheta(L) \neq 0$. Then $\vartheta(K*L) \neq 0$. 
\end{Conjecture}
The above conjecture relates to the following theorem of Gr\"unbaum \cite{Gru69}.
\begin{theorem}\label{t:grunbaum}
For $i=1,\ldots,m$, let $d_i \geq 0$ be integers and set $d = \sum_i {d_i}+m-1$. Define $K_i$ to be the $d_i$-skeleton of a $(2d_i+2)$-simplex. Then the complex, $$K = K_1 * K_2 * \cdots K_m,$$ is a $d$-complex that does not embed into $\Rspace^{2d}$.  
\end{theorem}

The above theorem follows from the above conjecture. The second conjecture is the following.
 
\begin{Conjecture}[Algebraic Menger Conjecture]
Let $K,L$ be simplicial complexes such that $\vartheta(K)\neq 0$ and $\vartheta(L) \neq 0$. Then $\vartheta(K\times L) \neq 0$. 
\end{Conjecture}

The above conjecture relates to the Menger conjecture \cite{Men29}, now the following theorem.

\begin{theorem}\label{t:menger}
	Let $K,L$ be two non-planar graphs, then $K \times L$ does not embed into $\Rspace^4$.
\end{theorem}

Since the van Kampen obstruction is complete for $d=1$, Theorem \ref{t:menger} is a consequence of the algebraic Menger conjecture.
The Menger conjecture was first proved by Ummel \cite{Umm78} for product of two graphs using advanced algebraic topology techniques. M. Skopenkov \cite{MSk03} gave an elegant geometric proof of the more general case of product of multiple graphs. The above conjecture has been posed also in \cite{ASk14}.

\paragraph{Related work}
The classical theorem of Kuratowski states that a graph is planar if and only if it doesn not contain a subgraph homeomorphic to $K_5$ or to $K_{3,3}$. During the 1930's van Kampen \cite{vKam33} and Flores \cite{Flo32,Flo33} proved that the complexes $\sigma_{2d+3}^d$ are not even continuously embeddable into $\Rspace^{2d}$. Here, the complex $\sigma_{2d+3}^d$ can be defined as all the $d$-simplices (and their faces) built using $2d+3$ vertices, i.e., the $K_i$ of Theorem \ref{t:grunbaum}. For $d=1$ this is $K_5$. This settled the question of existence of such complexes. Any $d$-complex embeds into $\Rspace^{2d+1}$. Flores \cite{Flo33,Ros60}, also proved that the multiple join of 3-vertices with itself $[3]*[3]*\cdots*[3]$ does not embed into the space of dimension equal to twice its dimension. These two families of complexes are also \textit{minimal}, in the sense that removing one simplex from them results in an embeddable complex. See also \cite{Wu74} Section III.3 Examples 3 and 4. 

Gr\"unbaum \cite{Gru69} proved Theorem \ref{t:grunbaum} by geometric methods and using the Borsuk-Ulam theorem. Moreover, he shows that the complexes appearing in Theorem \ref{t:grunbaum} are minimal. Ummel \cite{Umm73} gave a different prove of the Gr\"unbaum theorem but still he relies on the geometric arguments of \cite{Gru69}. Following \cite{Zak69}, He also proves the existence of an infinite number of minimal simplicial complexes of any dimension $d>1$. These are built by local modifications on a simplex.

In \cite{Sch93}, the author proves a generalization of Theorem \ref{t:grunbaum} by an algebraic method. However, his proof works only for the very restricted class of joins of nice complexes. As defined in \cite{Sch93}, a complex $K$ is \textit{nice}, if for any set of vertices of $K$, either the set defines a simplex, or the complement of the set defines a simplex and not both. One notes that the $\sigma_{2d+3}^d$ are nice.

The method of van Kampen for proving the non-embeddability of $\sigma_{2d+3}^d$ later was developed by Shapiro \cite{Sha57} and Wu \cite{Wu74} to proving the non-vanishing of  certain cohomology classes, called the embedding classes. The van Kampen obstruction therefore is an special case. They proved that for $d\neq 2$ the obstruction is complete for embeddings into $\Rspace^{2d}$. For $d= 2$ the van Kampen obstruction is not complete \cite{Freetal94}. As mentioned, there is no known algorithm for deciding PL embeddability of $2$-complexes into $\Rspace^4$. It is known that this problem is NP-hard \cite{Matetal10}, though. See the latter reference for more information on the complexity of deciding embeddability. See also \cite{ASk14} for more information on the embeddability problems.

Wu has used Smith classes as an obstruction for embeddability. He also proves that these Smith classes and embedding classes are equal. However, the Smith classes have a much more systematic definition. See the book \cite{Wu74} for complete details.

\paragraph{The method of the proof of the main result}
To prove Theorem \ref{t:smithclass} we must show that $A^m(K) \neq 0$ implies $A^{m+2}([3]*K) \neq 0$. Our task here is therefore, demonstrating that a certain cohomology class is non-zero. If we are computing with field coefficients, then, the cohomology is dual to homology, and this implies that for a non-zero cohomology class $[\phi]$ there is always a non-zero cycle $z$ for which, $\phi(z)\neq 0$. This $z$ can be taken as a certificate of the fact that $[\phi]$ is non-zero. Our aim is to find such a certificate for $A^{m+2}([3]*K)$. First we consider the case in which the mod 2 reduction of $A^m(K)$ does not vanish. In this case, indeed such a cycle $z$ can be defined. This special case is not necessary for the proof of the general case. With integer coefficients, there exist cohomology classes whose representatives vanish on all cycles but still are non-zero. The embedding class is always of this type. Therefore, it is not possible to certify its non-vanishing by a cycle. However, we prove in Lemma \ref{l:generalchain}, that such a certificate can be always found among chains with certain properties\footnote{This lemma is new to the author.}. Now the fact that $A^{m+2}([3]*K)\neq 0$ is proved by constructing the certificate chain from a certificate for $A^m(K)$. 

The origins of our definition of the certificate chains is the following simple case. Let us take as $K$ a graph that consists of two disjoint circles $z_1,z_2$. Let $K$ be embedded in a 3-space $\Rspace^3 \subset \Rspace^4$ such that the two circles are linked, see Fig. \ref{f:links}. We take three vertices $v,w,u$ and form the join $[3]*K$. We now prove that the embedding of $K$ given in Fig. \ref{f:links} does not extend to an embedding $f$ of $[3]*K$ into $\Rspace^4$ such that $f|_{\Rspace^3}$ coincides with the figure.

\begin{figure}\label{f:links}
	\centering
	
	\includegraphics[scale=0.4]{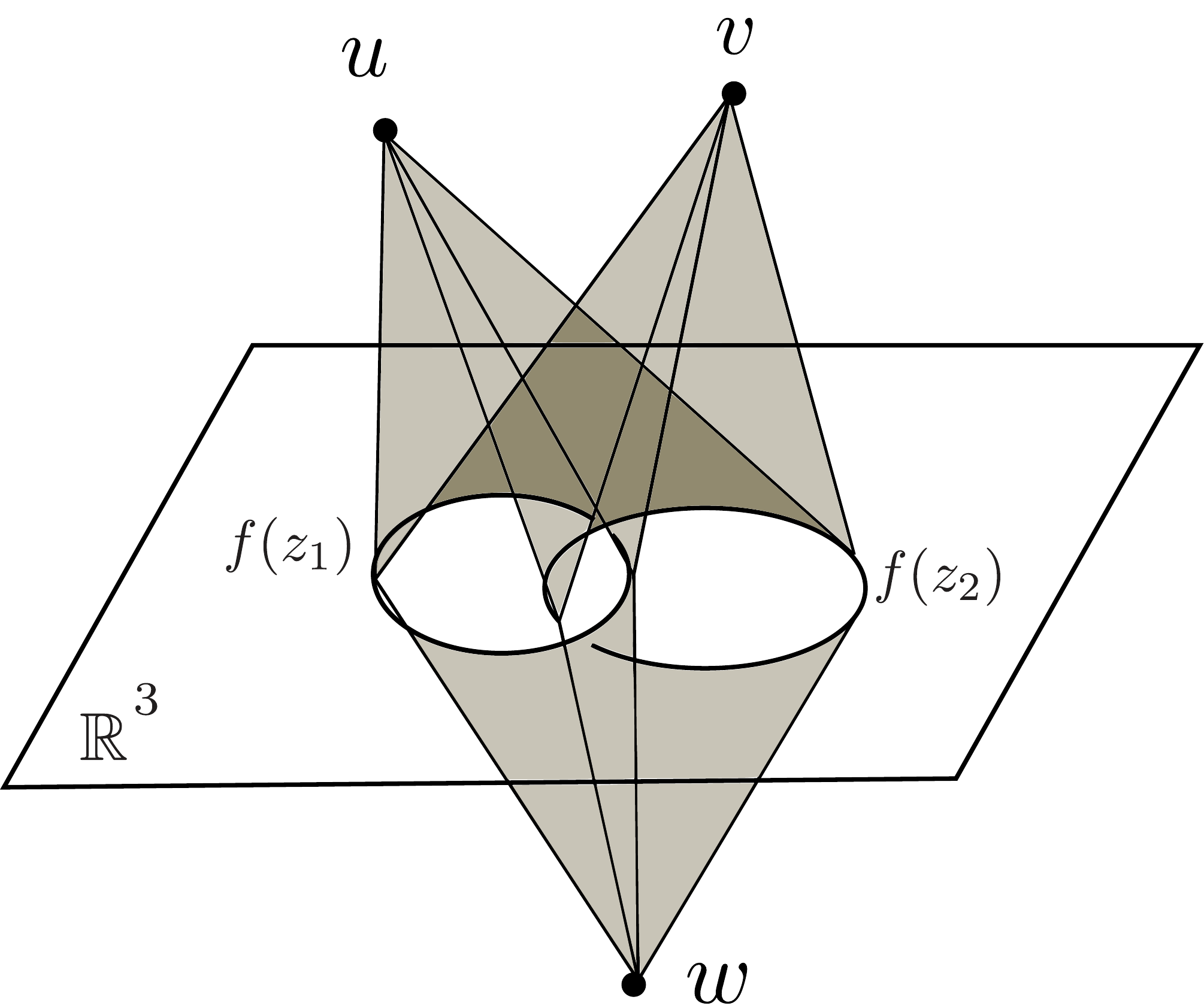}
	\caption{Schematic view of embedding of two circles as a link in $\Rspace^3,$ and, join of the image with three vertices.}
\end{figure}

We work with $\Zring_2$ coefficients. Let there be an extension that is an embedding and satisfying our conditions. Take the cycle $z=z_1 \times z_2 \in C(\Delta(K),\Zring_2)$. We first need two $3$-chains with boundary $z$. Let them be $c_1 = vz_1 \times z_2$ and $c_2 = wz_1 \times z_2$, where $vz_1$ means the join of $v$ and $z_1$, etc. It is easily seen that $\partial(c_1)=\partial(c_2)=z$. Hence, we have a $3$-cycle $c_1+c_2 = (v+w)z_1 \times z_2$. It is not difficult to observe that the $2$-cycle $(v+w)z_1$ and the $1$-cycle $z_2$ would be linked in $4$-space, in any extension satisfying our condition. This is equivalent to the fact that $(v+w)z_1 \times z_2$ has non-zero linking number with the diagonal in the product map into $\Rspace^8 \backslash \text{diagonal}$. Hence, there can be no $4$-chain $d\in C(\Delta([3]*K))$ with $\partial d = c_1+c_2$. But such a chain exists; take $d = vz_1\times uz_2 + wz_1 \times uz_2$. One easily checks that $\partial d = c_1+c_2$.

An important point to mention here is that the above cycle $z$ --our \textit{bad} cycle-- only is bad in the given embedding of $K$ in $\Rspace^3$. Indeed, $[3]*K$ is embeddable into $\Rspace^4$. However, the non-vanishing of the mod 2 van Kampen obstruction guarantees that such a bad cycle or chain always exist.

\paragraph{Outline of the paper}
In Section \ref{s:basic} we review some background definitions. Next, in Section~\ref{s:smith}, we present an overview of the Smith theory of complexes with a periodic action on them. We restrict ourselves to the simple case of free actions of $\Zring_2$, while the theory applies to a more general periodic actions with fix-points. In Section \ref{s:embeddingclasses} we define the embedding classes as obstructions to embeddability into Euclidean spaces. Finally, in Section~\ref{s:main}, we present the proof of our main result.

\section{Basic concepts}\label{s:basic}

In this section, we explain briefly the necessary background material and definitions needed in the later sections. In some places we favor simplicity of the exposition to the generality of notions. 

\paragraph{$\Zring_2$-Spaces}
For our purposes, it suffices to take as a ``space" a polyhedral complex. This is a cell complex where each (open) cell is the interior of a polyhedral ball in some euclidean space. Even more, we can restrict to complexes whose cells are either open simplexes or product of open simplices. The theory and results, of course, apply to a very general definition of a topological space, see \cite{Wu74}. We therefore consider our complex $K$ as an abstract set, each element of which is a polyhedral cell. We also denote the underlying space of such a complex by $K$. 

A \textit{$\Zring_2$-space} $(K,\tau)$, is a polyhedral cell complex $K$ with an action of non-trivial element of $\Zring_2=\{1, \tau\}$ on $K$ by a cellular homeomorphisms that permutes cells. We assume that the cellular action of $\tau$ is \textit{simple}, by which we mean, $\tau \sigma \neq \sigma$ for any cell $\sigma$. 

When there is no danger of confusion we write $K$ instead of $(K,\tau)$.

\paragraph{Transformations on product chain complexes}
Let $(K,\tau)$ be a $\Zring_2$-space. Whenever $K$ is a subcomplex of the cell complex $L \times L$, for $L$ a simplicial complex, we assume that the orientations on cells of $K$ are induced by fixed orientations on cells of $L$. It is well-known that, with this convention, for $\sigma_1 \times \sigma_2 \in K$ we have $$\partial(\sigma_1 \times \sigma_2) = \partial \sigma_1 \times \sigma_2 + (-1)^{d(\sigma_1)}\sigma_1\times \partial \sigma_2.$$ In the above and elsewhere $d(\sigma_1)$ denotes the dimension of the cell $\sigma_1$.  On such a complex $K \subset L\times L$, we are interested on the action of $\tau$ given by $$\tau (\sigma_1\times \sigma_2) = \sigma_2 \times \sigma_1.$$ It is not difficult to see that with the above convention on orientations we have the identity $$ \tau_\sharp(\sigma_1 \times \sigma_2) = (-1)^{d(\sigma_1)d(\sigma_2)}\sigma_2\times \sigma_1$$ where $\tau_\sharp$ is the homomorphism of the chain complex onto itself defined by $\tau$. The above formulas will be used extensively in this paper.

\paragraph{Fundamental domains}
Let $(K,\tau)$ be a $\Zring_2$-space. A set of cells $F\subset K$ is a \textit{fundamental domain} for the action of $\tau$ on $K$ if $F \cup \tau F = K$, $F \cap \tau F=\emptyset$.  That is, $F$ has exactly one cell from each orbit. Then, cells of a fixed dimension $d$ in $F$ form a fundamental domain for the $d$-cells of $K$. 
It is then clear that each chain $c \in C(K)$ can be written as $$c = \sum_{\sigma \in F} n_\sigma \sigma + n^\tau_\sigma \tau_\sharp \sigma,$$ where $n_\sigma, n_\sigma^\tau$ are integers. Note that forming a fundamental domain requires choosing a cell from each orbit.  

\paragraph{Deleted products and quotient complexes}
Let $L$ be a simplicial $d$-complex and consider the product complex $L\times L$. The \textit{deleted product} of $L$, denoted $\Delta(L)$, is a subcomplex of $L\times L$ defined as
$$\Delta(L)=\{\sigma_1 \times \sigma_2 \in L\times L; \sigma_1 \cap \sigma_2 = \emptyset \}.$$ In words, the deleted product is defined by cells that are products of vertex disjoint simplices. The action of $\tau$ is defined as above. This action is simple.

Let $(K,\tau)$ be a (simple) $\Zring_2$-space. One can form the quotient space $K /\tau$ by identifying cells that are mapped by $\tau$. Let $\pi: K \rightarrow K/\tau$ be this identification map.

It is standard that the quotient space $K/\tau$ has an abstract cellular structure defined as follows, see \cite{Wu74} II.1 Proposition 1. The cells of $K/\tau$ are orbits of cells of $K$. We write $[\sigma]=\{\sigma,\tau\sigma\}$ for the orbit of $\sigma$. The dimension of the cell $[\sigma]$ is dimension of $\sigma$ and a cell $[\sigma_1]$ is a face of a cell $[\sigma_2]$ if there is a representative for $[\sigma_1]$ that is a face of a representative of $[\sigma_2]$. To induce orientations on cells of $K/\tau$ we need to fix a fundamental domain $F$ for $K$. Then we have a homomorphism $\pi_{\sharp,i}: C_i(K) \rightarrow C_i(K/\tau)$ on $i$-dimensional chain groups, induced by the projection $\pi: K \rightarrow K/\tau$, $\pi(\sigma)=[\sigma]$. We can now define the boundary map for chains of $K/\tau$ as $$\partial ([\sigma])=\pi_\sharp (\partial \sigma),$$ where $\sigma \in F.$ Then, it is routine to check that $\partial\partial[\sigma]=0$ and that $\pi_\sharp$ is a chain map.

\paragraph{Equivariant maps}
Let $(K_1,\tau_1), (K_2,\tau_2)$ be two $\Zring_2$-spaces. A continuous function $f:K_1 \rightarrow K_2$ is ($\Zring_2$) equivariant if $f\tau_1=\tau_2f$. 

We now show that any $\Zring_2$-space $(K,\tau)$ can be mapped equivariantly into a sphere with antipodal action, $(S^{m},\tau)$, for some $m \geq 0$. This is fairly straightforward. We give the infinite-dimensional sphere $S^\infty$ a cell structure with two cells in each dimension, such that the antipodal map exchanges the cells.
Let $F^{(i)}$ denote a fundamental domain for the cells of dimension $i$. We start with $F^{(0)}$ and map them to arbitrary vertices of $S^\infty$. Then, we map $\tau F^{(0)}$ such that the map on 0-skeleton of $K$ is equivariant. Now consider arbitrary $i$ and assume we have already mapped the $(i-1)$-skeleton equivariantly. We take an $i$-cell $\sigma$. We have already mapped $\partial \sigma$, which is an $(i-1)$-sphere, into $S^{i-1}$. This map can be extended trivially onto $\sigma$ with image of $\sigma$ in any of the two closed hemi-spheres of $S^{i}$. The cell $\tau \sigma$ will be extended symmetrically. This finishes the definition of our equivariant map. 

It follows from the fundamental theorems in the theory of fibre bundles that the above map is unique up to equivariant homotopy \cite{Hus94}. This means that any other equivariant map of $(K,\tau)$ into $(S^{\infty}, \tau)$ is homotopic to the map we constructed with a homotopy consisted of equivariant maps.

\paragraph{Remark}
Observe that the map defined above maps a $d$-dimensional complex into $S^{m}$ with $m=d$. The smallest $m$ such that a continuous equivariant $f$ into $(S^m, \tau)$ exists is called the \textit{topological $\Zring_2$-index} of $(K,\tau)$ \cite{Mat08}. This index plays a major role in some important combinatorial problems, see \cite{Mat08}.    

\section{The Smith theory}\label{s:smith}
In this section, we review basics of the theory of P. A. Smith for the study of periodic transformations acting on a space \cite{RSm38,Smi33}. For more details and the general theory refer to \cite{Wu74,Nak56,Aepetal70}. 

\paragraph{Remark}The Smith theory when applied to the special case of complexes that are deleted products with the usual action of $\Zring_2$ on them, results in an alternative and more informative definition of the obstruction classes for embeddability into euclidean spaces, see \cite{Wu74}, Chapter~III. This is our main motivation for their study.
\paragraph{}
For simplicity of the exposition, we present the theory only for the case of actions of $\Zring_2$. However, the arguments given suggest the generalization to the case of the action of $\Zring_p$ for $p$ prime .

\subsection{Special homology groups of $(K,\tau)$}
Let $(K,\tau)$ be a (simple) $\Zring_2$-complex. For simplicity, we assume that $\tau$ is a cell-map that simply permutes the cells of each dimension. The theory, however, applies to a more general setting. The map $\tau$ induces a chain map on the integer chain complex $C(K)$, denote this chain map by $\tau_\sharp$. For any dimension $i$, consider the sequence
\begin{equation}\label{eq:sdexact}
    \begin{split}
        \cdots \rightarrow C_i(K)\xrightarrow{1+\tau_{i\sharp}} C_i(K) \xrightarrow{1-\tau_{i\sharp}} C_i(K) \xrightarrow{1+\tau_{i\sharp}} \cdots.
    \end{split}
\end{equation}

\begin{lemma}\label{lemma:homexact}
The sequence (\ref{eq:sdexact}) is exact.
\end{lemma}
\begin{proof}
Let $F=F^{(i)}$ be a fundamental region for the $i$-dimensional cells of $K$, and for $\sigma \in F$. Then any chain $x\in C_i(K,\Zring)$ can be written as
$$x=\sum_{\sigma_j \in F} (a_j \sigma_j + a'_j \tau_\sharp \sigma_j)$$
for some $a_j,a^\tau_j \in \Zring$. Since the action of $\Zring_2$ on $K$ is free the above presentation is unique.
If $(1+\tau_\sharp)(x)=0$, then for all $\sigma_j \in F$, $a_j+a'_j=0$.
If $(1-\tau_\sharp)(x)=0$, then for all $\sigma_j \in F$, $a_j-a'_j=0$. It follows that $a_j=a'_j$ in this case.

Now assume $x \in \ker(1+\tau_\sharp)$. Then we can find $b_j$ and $b'_j$ such that $a_j = b_j -b'_j$ and $a'_j = b'_j - b_j$. Let $y=\sum_{\sigma_j \in F} (b_j \sigma_j + b'_j \tau_\sharp \sigma_j)$.
We then have $(1-\tau_\sharp)(y)=x$. Therefore $\ker(1+\tau_\sharp) \subset (1-\tau_\sharp)(C_i(K))$. Since $(1+\tau_\sharp)(1-\tau_\sharp) = 0$, we deduce $\ker(1+\tau_\sharp) = (1-\tau_\sharp)(C_i(K))$.

Assume $x \in \ker (1-\tau_\sharp)$. Then $a_j=a'_j$ in the presentation of $x$. Set now $b_j=a_j$, $b'_j=0$ and define $y$ as before. We have then $(1+\tau_\sharp)(y)=x$. And similarly to the above it follows that $\ker(1-\tau_\sharp) = (1+\tau_\sharp)(C_i(K))$. Therefore the sequence is exact.
\end{proof}

\paragraph{}
Set $$C_i^\delta=C_i^\delta(K) = (1+\tau_\sharp)(C_i(K)) = \ker(1-\tau_{\sharp}).$$ The chains in $C_i^\delta$ are invariant chains under $\tau_\sharp$. It is easily seen that $\partial(C_i^\delta)\subset C_{i-1}^\delta$ and hence the groups $C_i^{\delta}$, $C_{-1}^i=0$, form a chain complex with the restriction of the same differential $\partial$. We call the elements of $C_i^\delta$ $i$-dimensional \textit{$\delta$-chains}. Similarly, we denote the cycles by $Z_i^\delta(K)\subset C_i^\delta(K)$ and call them $i$-dimensional \textit{$\delta$-cycles}, analogously for \textit{$\delta$-boundaries}. The homology groups of the chain complex $C^\delta(K,\tau)$ are called the \textit{special $\delta$-homology groups} of $K$ with respect to the involution $\tau$, and are written as $H_i^\delta(K) = H_i^\delta(K,\tau)$. It is not difficult to show that these homology groups are isomorphic with the usual homology of the quotient complex $K/\tau$, see \cite{Wu74}, Proposition 4 of II.2.

Analogously, we define the \textit{$s$-chain complex} by setting $$C^s_i = C^s_i(K) = (1-\tau_\sharp)(C_i(K))=\ker(1+\tau_\sharp) .$$ The chains in $C_i^s$ have the property that the sum of the coefficients of cells in any orbit of $\tau_\sharp$ is zero. Similarly to the $\delta$-chains, we define \textit{$s$-boundaries} and \textit{$s$-cycles}. The corresponding homology groups are denoted by $H_i^s(K)=H_i^s(K,\tau)$ and are called the \textit{special $s$-homology groups} of $K$ with respect to the involution $\tau$.  

\subsection{Special cohomology groups of $(K,\tau)$}
The $\delta$-cohomology groups of $(K,\tau)$, as well as the $s$-cohomology groups of $(K,\tau)$, are defined by ``dualizing" the definition of the corresponding special homology groups. Let $C^i = C^i(K,\Zring)=\text{Hom}(C_i(K),\Zring)$ be the group of $i$-dimensional cochains of $K$ with integer coefficients, and, denote the cochain complex of $K$ by $C^*(K)$. Let $\tau^\sharp$ denote the chain map induced by $\tau$ on the cochain complex $C^*(K)$. By definition $\tau^\sharp(\phi)(c)=\phi(\tau_\sharp(c))$ for any cochain $\phi$. For a fixed dimension $i$ consider the cochain complex

\begin{equation}\label{eq:sdexactco}
    \begin{split}
        \cdots \rightarrow C^i(K)\xrightarrow{1+\tau_{i}^{\sharp}} C^i(K) \xrightarrow{1-\tau_i^{\sharp}} C_i(K) \xrightarrow{1+\tau_i^{\sharp}} \cdots.
    \end{split}
\end{equation}

\begin{lemma}
The sequence (\ref{eq:sdexactco}) is exact.
\end{lemma}
\begin{proof}
Let $F$ and $\sigma_j$ be as in the proof of Lemma \ref{lemma:homexact}. Let $\phi \in C^i(K)$ be a cochain. 
If $(1+\tau^\sharp)(\phi)=0$, then for all $\sigma_j \in F$, $\phi((1+\tau_\sharp)\sigma_j)=\phi(\sigma_j+\tau_\sharp \sigma_j)=0$. Thus, for all $j$, $\phi(\sigma_j)+\phi(\tau_\sharp \sigma_j)=0$.
If $(1-\tau^\sharp)(\phi)=0$, then for all $\sigma_j \in F$, $\phi(\sigma_j)=\phi(\tau_\sharp \sigma_j)$.

Now assume $\phi \in \ker(1+\tau^\sharp)$. Then it is possible to find $\psi \in C^i(K)$ such that $\psi(\sigma_j)-\psi(\tau_\sharp\sigma_j)=\phi(\sigma_j)$, and, $\psi(\tau_\sharp\sigma_j)-\psi(\sigma_j)=\phi(\tau_\sharp\sigma_j)$. Then $(1-\tau^\sharp)(\psi)=\phi$. Hence, $\ker(1+\tau^\sharp) \subset (1-\tau^\sharp)(C^i(K))$. The other inclusion follows from
$(1+\tau^\sharp)(1-\tau^\sharp)=0$.

Now take $\phi \in \ker(1-\tau^\sharp)$. Then if we define $\psi$ by $\psi(\sigma_j)=\phi(\sigma_j)$, $\psi(\tau_\sharp\sigma_j)=0$, then, $(1+\tau^\sharp)(\psi)= \psi+\tau^\sharp(\psi)=\phi$. Hence, $\ker(1-\tau^\sharp) \subset (1+\tau^\sharp)(C^i(K))$. The other inclusion follows from $(1-\tau^\sharp)(1+\tau^\sharp)=0$.
\end{proof}

\paragraph{}
Set $$C^i_\delta=C^i_\delta(K) = (1+\tau^\sharp)(C^i(K)) = \ker(1-\tau^{\sharp}).$$ The cochains in $C^i_\delta$ are invariant under $\tau_\sharp$, i.e., they assign the same value to chains in an orbit of $\tau_{\sharp}$. It is easily seen that $\delta(C^i_\delta)\subset C^{i+1}_\delta$ and hence the groups $C^i_{\delta}(K)$ form a cochain complex with the restriction of the dual of the differential, $\delta = \partial^*$. We call the elements of $C^i_\delta$ $i$-dimensional \textit{$\delta$-cochains}. Similarly, we denote the cycles by $Z^i_\delta(K)\subset C^i_\delta(K)$ and call them $i$-dimensional \textit{$\delta$-cocycles}, and analogously for \textit{$\delta$-coboundaries}. The homology groups of this cochain complex are called the \textit{special $\delta$-cohomology groups} of $K$ with respect to the involution $\tau$ and are denoted by $H^i_\delta(K) = H^i_\delta(K,\tau)$. It is not difficult to show that these cohomology groups are isomorphic with the usual cohomology of the quotient complex, see \cite{Wu74}, Proposition 4 of II.2. 

Analogously, we define the \textit{$s$-cochain complex} by setting $$C_s^i = C_s^i(K)  = (1-\tau^\sharp)(C^i(K))= \ker(1+\tau^\sharp).$$ The cochains in $C^i_s$ have the property that the sum of their values on any orbit of $\tau_\sharp$ is zero. We shall use \textit{$s$-coboundaries} and \textit{$s$-cocycles} with the obvious meaning. The corresponding cohomology groups are denoted by $H^i_s(K)=H^i_s(K,\tau)$, and are called the \textit{special $s$-cohomology groups} of $K$ with respect to the involution $\tau$.    

\paragraph{Remark} Although the $\delta$-homology or $\delta$-cohomology of $(K,\tau)$ are isomorphic with the usual homology and the cohomology of the quotient complex, $K/\tau$, respectively, the relation of $s$-homology and $s$-cohomology of $(K,\tau)$ with those of the quotient complex is more complicated. In brief, the $s$-(co)homology with coefficient group the abelian group $G$ is isomorphic with (co)homology of $K/\tau$ with respect to a certain boundary operator and with the coefficient group $(G\oplus G)^0$, consisting of all elements $g_1\oplus g_2 \in G \oplus G$ satisfying $g_1+g_2=0$, see Proposition 5 in Section II.2 of \cite{Wu74}. 

\subsection{The Smith-Richardson exact sequences}

Consider the following short exact sequence of chain complexes.

\begin{equation}
    \begin{split}
        0 \rightarrow C^\delta(K) \subset C(K) \xrightarrow{1-\tau_\sharp} C^s(K)\rightarrow 0
    \end{split}
\end{equation}
The standard long exact sequence defined by this short exact sequence of chain complexes is as follows, where $\iota$ denotes the inclusion map.

\begin{equation}\label{e:dsrexact}
    \begin{split}
\cdots \rightarrow H_{i}^\delta(K)&\xrightarrow{\iota_*} H_{i}(K) \xrightarrow{(1-\tau)_*}\\& H^s_{i}(K) \xrightarrow{\mu^s_{i}} H^\delta_{i-1}(K) 
\xrightarrow{\iota_*} H_{i-1}(K) \rightarrow \\& \cdots \rightarrow H_1^s(K) \xrightarrow{\mu^s_1} H_0^\delta(K) \xrightarrow{\iota_*} H_0(K) \xrightarrow{(1-\tau)_*} H_0^s(K) \rightarrow 0.
\end{split}
\end{equation}
The homomorphisms $\mu^s_i$ are defined as follows. Let $z$ be an $i$-$s$-cycle. Then $z \in \ker(1+\tau_\sharp)=\text{im}(1-\tau_\sharp)$, hence there exists a $c$ such that $(1-\tau_\sharp)c=z$. Then $(1-\tau_\sharp) \partial c=0$, so that $\partial c \in Z^\delta_{i-1}(K)$, and,  
$\mu^s_i([z])=[\partial c]$. The standard theory of long exact sequences shows that $[\partial c]$ is independent of the choices of $c$ and $z$. 

\paragraph{}
We can define another short exact sequence of chain complexes:
\begin{equation}
    \begin{split}
        0 \rightarrow C^s(K) \subset C(K) \xrightarrow{1+\tau_\sharp} C^\delta(K)\rightarrow 0.
    \end{split}
\end{equation}
Then the long exact sequence of homology groups for this short exact sequence of chain complexes is
\begin{equation}\label{e:ssrexact}
    \begin{split}
\cdots \rightarrow H_{i}^s(K)&\xrightarrow{\iota_*} H_{i}(K) \xrightarrow{(1+\tau)_*} \\&H^\delta_{i}(K) \xrightarrow{\mu^\delta_{i}} H^s_{i-1}(K) 
\xrightarrow{\iota_*} H_{i-1}(K) \rightarrow\\& \cdots \rightarrow H_1^\delta(K) \xrightarrow{\mu^\delta_1} H_0^s(K) \xrightarrow{\iota_*} H_0(K) \xrightarrow{(1+\tau)_*} H_0^\delta(K) \rightarrow 0.
\end{split}
\end{equation}

The homomorphisms $\mu^\delta_i$ have analogous definition to the $\mu^s_i$.

The sequences (\ref{e:ssrexact}) and (\ref{e:dsrexact}) are called the \textit{Smith-Richardson exact homology sequences}. 

Observe that the homomorphisms $\mu^\delta_j, \mu^s_j$ can be composed together. For $j,k$ in a suitable range, set $$\mu^\delta_{j,k}= \mu^\rho_{j-k}\cdots\mu^s_{j-1}\mu^\delta_j : H^\delta_j(K) \rightarrow H^{\rho}_{j-k}(K)$$
where $\rho$ is $s$ or $\delta$ if $k$ is odd or even, respectively. Analogously, define $\mu^s_{j,k}$.
The homomorphisms $\mu^\delta_{j,k}, \mu^s_{j,k}$ are called \textit{the Smith special homomorphisms} of $(K,\tau)$.

\paragraph{}
The above can be dualized to obtain two long exact sequences for cohomology called the \textit{Smith-Richardson cohomology sequences}, namely,
\begin{equation}\label{e:dsrcoexact}
    \begin{split}
\cdots \xleftarrow{\mu_\delta^{i+1}} H^{i}_\delta(K)&\xleftarrow{\iota^*} H^{i}(K) \xleftarrow{(1-\tau)^*}\\& H_s^{i}(K) \xleftarrow{\mu_\delta^{i}} H_\delta^{i-1}(K) 
\xleftarrow{\iota^*} H^{i-1}(K) \leftarrow \\& \cdots \leftarrow H^1_s(K) \xleftarrow{\mu_\delta^1} H^0_\delta(K) \xleftarrow{\iota^*} H^0(K) \xleftarrow{(1-\tau)^*} H^0_s(K)
\end{split}
\end{equation}
and
\begin{equation}\label{e:ssrcoexact}
    \begin{split}
\cdots \xleftarrow{\mu_s^{i+1}} H^{i}_s(K)&\xleftarrow{\iota^*} H^{i}(K) \xleftarrow{(1+\tau)^*}\\& H_\delta^{i}(K) \xleftarrow{\mu_s^{i}} H_s^{i-1}(K) 
\xleftarrow{\iota^*} H^{i-1}(K) \leftarrow \\& \cdots \leftarrow H^1_\delta(K) \xleftarrow{\mu_s^1} H^0_s(K) \xleftarrow{\iota^*} H^0(K) \xleftarrow{(1+\tau)^*} H^0_\delta(K).
\end{split}
\end{equation}

The homomorphisms $\mu_\delta^i$ are defined as follows. Let $\phi \in Z^{i-1}_\delta(K)$ be an $(i-1)$-$\delta$-cocycle. Then by definition $(1-\tau^\sharp)\phi=0$, and, there exists a $\psi \in C^{i-1}_\delta(K)$ such that $(1+\tau^\sharp)\psi=\phi$. Then $(1+\tau^\sharp)\delta \psi=0$, hence, $\delta \psi \in Z^{i}_s(K)$. We have $\mu_\delta^i([\phi]) = [\delta \psi]$. The homomorphisms $\mu_s^i$ are defined analogously.

Set $$\mu_\delta^{j,k}= \mu_\rho^{j+k}\cdots\mu_s^{j+2}\mu_\delta^{j+1} : H_\delta^j(K) \rightarrow H_{\rho}^{j+k}(K)$$
where $\rho$ is $s$ or $\delta$ if $k$ is odd or even, respectively. Analogously define $\mu_s^{j,k}$.
The homomorphisms $\mu_\delta^{j,k}, \mu_s^{j,k}$ are also called \textit{the Smith special homomorphisms} of $(K,\tau)$.

\paragraph{}
It turns out that, with current methods, the cohomology special homomorphism are far more useful for applications than the homological ones, as can be seen in the rest of this paper.

\subsection{The Smith special classes and special index}

Consider the 0-cochain $ 1 \in C^0(K)$ defined as $1(v)=1 \in \Zring$ for any vertex $v$ of the complex $K$. This is a cocycle and clearly a $\delta$-cocycle. Denote the $\delta$-cohomology class of $1$ by $\textbf{1}$. We can apply the special Smith homomorphisms to $\textbf{1}$. The $\delta$- or $s$-cohomology classes $$A^k=A^k(K,\tau):=\mu^{0,k}_\delta(\textbf{1})$$ are called the \textit{Smith special classes} of the system $(K,\tau)$. 

Let $n$ be the least integer, if it exists, such that $A^n(K,\tau)=0$. Then $n$ is called the \textit{special Smith index} of the system $(K,\tau)$ and denoted $I(K,\tau)$. If no such $n$ exists then we say that the special Smith index is infinite. 

\paragraph{}
As usual for homological properties of spaces we have the following important proposition.
\begin{proposition}\label{p:equivmap}
Let $(K,\tau)$ and $(K', \tau')$ be two $\Zring_2$-spaces and $f:K \rightarrow K'$ a $\Zring_2$-equivariant map. Then we have, for $\rho = \delta, s$ the induced homomorphisms
$$f^\rho_*: H_i^\rho(K,\tau) \rightarrow H_i^\rho(K', \tau')$$
on special homology groups, and,
$$ f^*_\rho: H_\rho^i(K',\tau') \rightarrow H^i_\rho(K,\tau)$$
on special cohomology groups. These homomorphisms commute with the corresponding homomorphisms of the Smith-Richardson sequences,
and in particular, with the special homomorphisms $\mu_i^\rho$, $\mu^i_\rho$.
\end{proposition}

We refer for the proof of the above theorem to the standard theory of the long exact sequences. 

\paragraph{}
As a consequence of Proposition \ref{p:equivmap} we have for an equivariant map $f$
\begin{equation}\label{e:functor}
f^*_\rho(A^k(K',\tau'))= A^k(K,\tau)
\end{equation}
where $\rho=s$ if $k$ is odd and $\rho=\delta$ if $k$ is even. It follows that when such a map $f$ exists then 
$$ I(K,\tau) \leq I(K',\tau').$$

\paragraph{Computing the special Smith index}
Here we note that the special Smith index of a simplicial $\Zring_2$-complex $(K,\tau)$ can be computed in polynomial time. Observe that we can compute the bases for the spaces $C^i_s(K)$, $C^i_\delta(K)$ as kernels of two linear homomorphism in polynomial time. To compute the classes $A^n(K,\tau)$ and the special Smith index, we start with the cocycle $1$ and check if $[1]\neq 0$. If so, we compute $\mu^1_\delta(1)$. This can be done by first finding a $\psi$ such that $(1+\tau^\sharp) \psi=\phi$, $\phi=1$. This is a system of linear equations over the integers and the cochain $\psi$ can be computed in polynomial time \cite{KB79,CC82,Ili89,Sto96}. Next compute $\delta \psi$. We need to check if $[\delta \psi] = 0$ in  $H^2_s(K)$. This can be done also in polynomial time by checking if $\delta \psi$ lies in $\delta(C^{1}_s)$. we continue by replacing $\phi$ with $\delta \psi$, and $\delta$ with $s$. We stop when $[\delta \psi] = 0$.

Here we remark that for the application of the (special) Smith classes to embeddability problems there exist more efficient methods of computation, see Section \ref{s:embeddingclasses}.

\subsection{(Co)Homology of the quotient complex}
Let $\pi:K \rightarrow K/\tau$ be the projection map, see Section \ref{s:basic}. We denote by $[\sigma]$ the cell of $K/\tau$ that is the image of cells $\sigma$ and $\tau \sigma$. The projection defines two chain maps and two cochain maps as follows. The first chain and cochain maps are the ordinary ones defined by $\pi$. The homomorphism $$\pi_\sharp: C(K) \rightarrow C(K/\tau)$$ is defined as follows. Take a fundamental domain $F$ of $(K,\tau)$ and for $\sigma_j \in F$, define $\pi_\sharp(\sigma_j)=\pi_\sharp(\tau_\sharp \sigma_j)=[\sigma_j]$. Moreover,  $$\pi^\sharp: C^*(K/\tau) \rightarrow C^* (K)$$ is the dual of $\pi_\sharp$. If $\phi=\pi^\sharp(\psi)$, then $\phi$ assigns the same value $\psi([\sigma])$ to $\sigma$ and $\tau_\sharp\sigma$, and hence, $(1-\tau^\sharp)(\phi)=0$ and $\phi$ is a $\delta$-cochain. Thus $\pi^\sharp(C^*(K/\tau))\subset C_\delta^*(K)$. On the other hand, $\pi^\sharp$ maps onto $C^*_\delta(K)$, and is an isomorphism.

The second pair of chain and cochain maps is defined in terms of the inverse projection; the so called transfer homomorphisms, see \cite{Hat02} Section 3.G.  The chain map $$\bar{\pi}_\sharp: C(K/\tau)\rightarrow C^\delta(K)$$ is defined by ${\bar\pi}_\sharp([\sigma]):=\sigma+\tau_\sharp \sigma=(1+\tau_\sharp)\sigma$, for a generator $[\sigma]$ of $C(K/\tau)$. It can be seen that this map is also an isomorphism onto $C^\delta(K)$, see Proposition 4 of Section II.2 of \cite{Wu74} \footnote{Notice the differences between here and Proposition2 of \cite{Wu74}. It seems that in Proposition4 of Section II.2 and Proposition 12 of Section II.3 by $\bar{\pi}^\sharp_{(d)}$ is meant $(\pi^\sharp)^{-1}$ rather than the dual of $\bar{\pi}_\sharp$. Otherwise, we would have for instance in Proposition 12 of II.3, $\bar{\pi}^\sharp 1_K = 2 \cdot 1_{K/\tau}$.}. In brief,

\begin{equation}
\begin{split}
    &{\bar{\pi}}_\sharp: C(K/\tau) \cong C^\delta(K)\\
    &{\pi}^\sharp: C^*(K/\tau) \cong C^*_\delta(K). 
    \end{split}
\end{equation}

The above implies,
\begin{equation}
\begin{split}
    &{\bar{\pi}}_*: H(K/\tau) \cong H^\delta(K)\\
    &{\pi}^*: H^*(K/\tau) \cong H^*_\delta(K). 
    \end{split}
\end{equation}

The second cochain map is the dual map to $\bar{\pi}_\sharp$, denoted $\bar{\pi}^\sharp$. In more detail, $$\bar{\pi}^\sharp: C^*(K) \rightarrow C^*(K/\tau)$$ is defined by $$\bar{\pi}^\sharp(\phi)([\sigma])=(1+\tau^\sharp)(\phi)(\sigma).$$ We can write more compactly $$\bar{\pi}^\sharp = (\pi^\sharp)^{-1}(1+\tau^\sharp).$$ 

Note that the homomorphism $({\pi^\sharp})^{-1}: C^*_\delta(K) \rightarrow C^*(K/\tau)$ is defined by $${(\pi^\sharp)}^{-1}(\phi)([\sigma]) = \phi(\sigma)=\phi(\tau_\sharp \sigma).$$

These four homomorphisms satisfy the following important identities.

\begin{equation}
\begin{split}
    &\bar{\pi}_\sharp \pi_\sharp = (1+\tau_\sharp), \; \pi_\sharp \bar{\pi}_\sharp = 2\cdot 1,\\
    &\pi^\sharp \bar{\pi}^\sharp = (1+\tau^\sharp), \; \bar{\pi}^\sharp \pi^\sharp = 2 \cdot 1,
    \end{split}
\end{equation}
where in the above $1$ denotes the corresponding identity map. The map $\bar{\pi}^\sharp$ is used to define reduced Smith classes and index.

\subsection{The Smith classes and index}
We study here the (reduced) Smith classes and index. It turns out that the obstructions to embeddability of simplicial complexes into Euclidean spaces are in fact certain Smith classes. Refer to \cite{Wu74} for complete details.

As was observed, the special $\delta$-(co)homology of a $\Zring_2$-space is isomorphic with the ordinary (co)homology of the quotient complex. Such a description of special $s$-(co)homology is not possible. However, we can, to some extent, recover the $s$-cohomology of $(K,\tau)$ if we use the coefficient group $\Zring_2$ for the quotient complex. This leads to reduced Smith classes that have many useful properties and are defined in terms of the ordinary cohomology of the quotient complex $K/\tau$ over the coefficient groups $\Zring$ and $\Zring_2$. This will then allow us to use tools and results available in the ordinary homology theory in our arguments for Smith classes.

\paragraph{}
Let $r_2: C^*(K/\tau) \rightarrow C^*(K/\tau,\Zring_2)$ be the reduction mod 2 of cochains. We start by defining a homomorphism $$p: H^i_s(K) \rightarrow H^i(K/\tau, \Zring_2)$$ as follows. Let $\phi \in Z^i_s(K)$ and take $\psi \in Z^{i}(K)$ such that $(1-\tau^\sharp)\psi=\phi$. Set $$p([\phi])=[r_2 \bar{\pi}^\sharp(\psi)].$$ We claim that $p([\phi])$ does not depend on the choices of $\phi$ and $\psi$. 

Assume $\phi_1,\psi_1$ satisfy $(1-\tau^\sharp)\psi_1=\phi_1$ and $\phi_1=\phi+\delta \theta$ for an $(i-1)$-$s$-cochain $\theta$. 
Then, calculating mod 2, $r_2(\bar{\pi}^\sharp(\psi_1))-r_2(\bar{\pi}^\sharp(\psi))=r_2(\bar{\pi}^\sharp(\psi_1-\psi))=
r_2((\pi^\sharp)^{-1}(1+\tau^\sharp)(\psi_1-\psi))= r_2((\pi^\sharp)^{-1}(1-\tau^\sharp)(\psi_1-\psi))= r_2((\pi^\sharp)^{-1}\delta \theta)=\delta(r_2 (\pi^\sharp)^{-1} \theta ).$ The claim is proved.

\paragraph{}
We are now ready to define the reduced Smith classes. Take $1 \in C^0(K/\tau)$. Note that the homomorphism $\bar{\pi}^\sharp$ is surjective. Take $\phi_0 \in C^0(K)$ with $\bar{\pi}^\sharp(\phi_0)=1$. In words, $\phi_0$ assigns 1 to only one vertex in any orbit of $\tau$, hence, is 1 on a fundamental domain of the 0-simplices. We have $(1+\tau^\sharp)(\delta \phi_0)=\delta (1+\tau^\sharp) \phi_0=\delta \pi^\sharp \bar{\pi}^\sharp \phi_0= \delta \pi^\sharp 1=\pi^\sharp \delta 1=0$. Therefore, $\delta \phi_0$ is an $s$-cocycle. Hence, there exists a cocycle $\phi_1$ such that $$(1-\tau^\sharp)\phi_1=\delta \phi_0.$$ We have $(1-\tau^\sharp) \delta \phi_1 =  \delta (1-\tau^\sharp) \phi_1= \delta \delta \phi_0=0$. Hence, $\delta \phi_1$ is a $\delta$-cocycle and there exists a $\phi_2$ such that $$(1+\tau^\sharp) \phi_2= \delta \phi_1.$$ Continuing in this manner, we obtain a sequence $\phi_0, \phi_1, \phi_2, \ldots$, such that, $\phi_i \in C^i(K)$, and the $\phi_i$ satisfy $(1+(-1)^i\tau^\sharp) \phi_i = \delta \phi_{i-1},$ and, $\bar{\pi}^\sharp(\phi_0)=1$. This sequence is called a \textit{resolution} for the cocycle $1 \in C^0(K/\tau)$. Now define the \textit{(reduced) Smith classes} by $$A^k_r(K, \tau)=[\bar{\pi}^\sharp(\phi_k)]\in H^k(K/\tau)$$ for $k$ even, and, for $k$ odd define $$A^k_r = p([\phi_k])= [r_2 (\bar{\pi}^\sharp(\phi_k))] \in H^k(K/\tau, \Zring_2).$$

\begin{lemma}
The classes $A^k_r(K,\tau)$ do not depend on the chosen resolution of $1 \in C^0(K/\tau)$.
\end{lemma}

\begin{proof}
Let $\psi_0, \psi_1, \psi_2, \ldots $ be another resolution of the cocycle $1 \in C^0(K/\tau)$. Therefore, the $\psi_i$ satisfy
$\bar{\pi}^\sharp \psi_0=1$ and $(1+(-1)^i(\tau^\sharp)) \psi_i=\delta \psi_{i-1}$ for $i\geq 1$. Assume $i$ odd. Then $r_2(\bar{\pi}^\sharp(\phi_i))-r_2(\bar{\pi}^\sharp(\psi_i))= r_2(\bar{\pi}^\sharp(\phi_i-\psi_i))= r_2(({\pi^\sharp})^{-1}(1+\tau^\sharp) (\phi_i-\psi_i)) = r_2((\pi^{\sharp})^{-1} (1+\tau^\sharp) (\phi_i - \psi_i)- 2(\pi^\sharp)^{-1}\tau^\sharp(\phi_i-\psi_i))= r_2((\pi^\sharp)^{-1}(1-\tau^\sharp) (\phi_i-\psi_i)) = r_2((\pi^\sharp)^{-1}(\delta(\phi_{i-1}-\psi_{i-1})))= \delta(r_2((\pi^\sharp)^{-1}(\phi_{i-1}-\psi_{i-1})))$, therefore $\bar{\pi}^\sharp(\phi_i)$ and $\bar{\pi}^\sharp(\psi_i)$ belong to the same cohomology class mod 2. A simpler calculation shows that in the case $i$ even the cochains $\bar{\pi}^\sharp(\phi_i)$ and $\bar{\pi}^\sharp(\psi_i)$ belong to the same integer class.

\end{proof}

\begin{proposition}
For the classes $A^n_r(K,\tau)$ and special classes $A^n(K,\tau)$ we have, for $n$ even
$$A_r^n(K,\tau)= (\pi^*)^{-1} A^n(K,\tau),$$ and for $n$ odd, $$A_r^n(K,\tau) = r_2^* (\pi^*)^{-1} A^n(K, \tau).$$
\end{proposition}
\begin{proof}
Let $\phi_0, \phi_1, \ldots$ be a resolution for $1 \in C^0(K/\tau)$. Then by definition, $\phi_0$ is a cochain such that $\bar{\pi}^\sharp \phi_0=1$ and  $(1+(-1)^i(\tau^\sharp)) \phi_i=\delta \phi_{i-1}$ for $i\geq 1$. We have $1= (1+\tau^\sharp)\phi_0$ and $(\pi^\sharp)^{-1} (1+\tau^\sharp) \phi_0 = \bar{\pi}^\sharp \phi_0 = 1.$ Hence, the claim is true for $n=0$. Next consider $n=1$. By definition $A^1=\mu^0([1])=[\delta \phi_0]$. On the other hand, 

\begin{equation*}
\begin{split}
A^1_r &= p([\phi_1]) = [r_2(\bar{\pi}^\sharp \phi_1 )]\\&=[r_2 (\pi^\sharp)^{-1} (1+\tau^\sharp)\phi_1]=[r_2 (\pi^\sharp)^{-1}(1-\tau^\sharp)\phi_1]=[r_2(\pi^\sharp)^{-1}(\delta \phi_0)]\\& =r_2^*(\pi^*)^{-1}[\delta \phi_0] = r_2^*(\pi^*)^{-1}A^1.
\end{split}
\end{equation*}
Hence, the claim is true for $n=1$. And similarly for any $n$.
\end{proof}

\paragraph{Remark} For many interesting and useful properties of the reduced classes see \cite{Wu74} Chapter II. Here we just observe that, by definition, the reduced classes are just isomorphic images of the non-reduced classes in even dimensions, in particular, for the van Kampen obstruction. From this (and the relation of cochains in a resolution) follows that the reduced index is at most one smaller than the non-reduced index.

\subsection{Example}\label{s:example}
We present in this section an important example of computing the Smith classed and index. This example is essentially taken from \cite{Wu74}. The interested reader should contact Chapter II of this book for more index computations.

Let $S^n$, $n\geq 0$, denote the unit sphere in $\Rspace^{n+1}$. On this sphere the antipodal map $\tau$ acts by negation, $\tau(x)=-x$. Take a cell decomposition of $S^n$ which is invariant under the action of $\tau$, that is, $\tau$ acts by a map that takes cells to cells of the same dimension. We can moreover assume that there are only two cells $\sigma^i, \sigma'^i$ of any dimension $i\geq 0$. The quotient space of this action is famously the real projective space $\Rspace P^n$. 

We now oriented the cells as follows. We assume that the vertex $\sigma^0$ is positive. Inductively, let $\sigma^i$ be the chain of dimension $i$ defined by an arbitrary $i$-cell but with an orientation that induces the same orientation on $\sigma^{i-1}$ as the orientation of $\sigma^{i-1}$. This gives a set of coherently oriented cells $\sigma^i$. Now assign orientations to the $\sigma'^i$ such that $\tau_\sharp \sigma^i = \sigma'^i$. Since $\tau$ reverses the orientation of every direction, $\tau_\sharp$ reverses the orientation of a cell when $i$ is odd. It follows then that
$$S^i = \sigma^i + (-1)^{i+1}\sigma'^i,$$
is an oriented $i$-sphere, whose orientation is induced from $\sigma^{i+1}$. We can deduce $$\partial \sigma^i = (-1)^{i} \partial \sigma'^{i} = \sigma^{i-1} + (-1)^{i} \sigma'^{i-1}.$$

Let $\gamma^i$ denote the cell of $S^n/\tau$ which is the projection of $\sigma^i, \sigma'^i$. Then, $\partial \gamma^i = 0$ if $i$ is odd and $\partial \gamma^i=2 \gamma^{i-1}$ if $i$ is even. Let $\gamma_i$ denote the $i$-cochain that assigns the value $1$ to $\gamma^i$. It is easily seen that $\delta \gamma_i=0 \mod{2}$ and $H^i(S^n/\tau, \Zring_2)=\Zring_2$ and is generated by $\gamma_i$. For $i$-even, we also have $H^i(S^n/\tau,\Zring)=\Zring_2$ and is generated by $U^i$ that assigns the value 1 on the $i$-dimensional cell $\gamma^i$, and for $i$ odd, there are no non-trivial cycles of dimension $i$ (see below). Note that here the generator of $H^i(S^n/\tau, \Zring_2)$ is $r_2U^i$ for $i$ even.

In order to compute the classes $A^i_r$ we need to find a resolution of $1 \in C^0(S^n/\tau)$. Let the cochain $\sigma_i \in C^i(S^n)$ be defined by $\sigma_i(\sigma^i)=1, \sigma_i(\sigma'^i)=0$, $0 \leq i \leq n$.  That is, the $\sigma_i$ are 1 on a fundamental domain. We have $\bar{\pi}^\sharp(\sigma_0)=1$. Moreover, $\delta(\sigma_i) (\sigma^{i+1})= \sigma_i \partial(\sigma^{i+1})=\sigma_i(\sigma^i+(-1)^{i+1}\sigma'^i) = 1$, and, $$\delta(\sigma_i)(\sigma'^{i+1}) =  \sigma_i \partial (\sigma'^{i+1})=(-1)^{i+1} \sigma_i\partial(\sigma^{i+1})=(-1)^{i+1}.$$ 
Therefore, $$\delta(\sigma_i) = \sigma_{i+1}+(-1)^{i+1}\tau^\sharp \sigma_{i+1}.$$ Consequently, the $\sigma_i$ constitute, by definition, a resolution of $1 \in C^0(S/\tau)$. Note that $\delta(\sigma_i)$ is a $\delta$- or $s$-ococycle that is a coboundary. It follows that for $i$ even $$A^i_r(S^n,\tau) = [\bar{\pi}^\sharp(\sigma_i)]=[\gamma_i],$$ and for $i$ odd, $$A^i_r(S^n,\tau) = [r_2 \bar{\pi}^\sharp(\sigma_i)]=[r_2\gamma_i].$$ 
It follows that $I_r(S^n,\tau)= n+1$.

Note that $1 \in C^0(S^n)=(1+\tau^\sharp)\sigma_0$. By definition and the above, $$A^i=[\delta \sigma_{i-1}]= [(1+(-1)^{i}\tau^\sharp)\sigma_{i}].$$ Using the formula for $\partial \sigma_i$, an easy calculation shows that the cocycle $(1+(-1)^{i}\tau^\sharp)\sigma_{i}$ cannot be a coboundary in its corresponding special cohomology group. It follows that, $A^i(S^n,\tau)\neq 0$ if and only if $0 \leq i \leq n$, and, $I(S^n, \tau)=n+1.$

\subsection{The Smith classes of the deleted products}
We study the Smith classes of deleted products of simplicial complexes, with the action of $\Zring_2$ on them which exchanges the factors. They have been defined by Wu \cite{Wu74} as an obstruction to embedding a simplicial complex into a Euclidean space. It turns out that the van Kampen obstruction is a special case of these classes. The van Kampen obstruction has a more natural definition and is the subject of the next section.

Let $K$ be a simplicial complex and let $\Delta(K)$ denote its deleted product. The complex $\Delta(K)$ plays the role of the $\Zring_2$-complex $K$ of the previous sections. The space $\Delta(K)$ can be considered as a polyhedral or cellular complex, where each cell is a product of two simplicies. Let $\tau$ be the cellular involution that sends $\sigma_1 \times \sigma_2 \in \Delta(K)$ to $\sigma_2 \times \sigma_1$. This map permutes the cells. With this action of $\Zring_2$ on $\Delta(K)$, we can define, for each $n \geq 0$, the special Smith classes $A^n(\Delta(K), \tau)$ and the (reduced) Smith classes $A^n_r(\Delta(K), \tau)$. We write $\Phi^n(K)=A^n(\Delta(K),\tau)$, and analogously we define the $\Phi^n_r(K)$.

The following theorem, easily proved, shows that these classes can serve as obstructions to embeddability.
\begin{Theorem}
Let $f: K \rightarrow L$ be an embedding. Then, for each $n\geq 0$, if $\Phi^n(L)=0$ then $\Phi^n(K)=0$. Similarly, if $\Phi^n_r(L)=0$ then $\Phi^n_r(K)=0$. 
\end{Theorem}

It is easily seen that the deleted product $\Delta(\Rspace^N)$ is equivariantly homotopic to the $(N-1)$-dimensional sphere $S^{N-1}$ and the antipodal action on it. From the calculations in \ref{s:example}, it follows that, for $n \geq 0$,

\begin{equation}
    \begin{split}
    \Phi^n(\Rspace^N)&=0\; \Leftrightarrow \; n \geq N,\\ 
    \Phi^n_r(\Rspace^N) &=0 \; \Leftrightarrow \; n \geq N.
    \end{split}
\end{equation}
Consequently, if the simplicial complex $K$ is embeddable into $\Rspace^N$, then,
\begin{equation}
    \begin{split}
    \Phi^n(K)&=0\;\; \text{for}\;\; n \geq N,\\ 
    \Phi^n_r(K) &=0 \;\; \text{for}\;\; n \geq N.
    \end{split}
\end{equation}

\section{The embedding classes and the van Kampen obstruction}\label{s:embeddingclasses}
In this section, we present an overview of the embedding classes as defined in Shapiro \cite{Sha57} and Wu \cite{Wu74} Chapter V.
Let $K$ be a $d$-dimensional simplicial complex, and $f:K \rightarrow \Rspace^m$ an arbitrary (PL) map. Assume $m\leq 2d$, otherwise $f$ can be made into an embedding by an arbitrary small perturbation. For non-negative integers $l,k$ such that $l+k=m$, let $\sigma^l, \sigma^k$ be two disjoint simplices of respective dimensions. We can consider the intersection number of two singular chains $f(\sigma^l)$ and $f(\sigma^k)$ in $\Rspace^m$ when defined. Let $\inter{(f(\sigma^l), f(\sigma^k))}$ denote the algebraic intersection number of the two chains. In general, the intersection numbers for two chains $A$ and $B$, satisfy  $$\inter{(A,B)}=(-1)^{\dim(A)\dim(B)}\inter(B,A).$$ 

\paragraph{}
For $f$ in general position, we can assume that two disjoint simplices $\sigma^l, \sigma^k$, with $l+k=m$, satisfy $f(\sigma^l) \cap f(\partial \sigma^k)=f(\partial \sigma^l) \cap f(\sigma^k)= \emptyset$. Then, the intersection numbers can be used to define an $m$-cochain $\vartheta^m_f(K)\in C^m(\Delta(K))$ by

$$ \vartheta^m_f(\sigma^l \times \sigma^k)= (-1)^l \inter{(f(\sigma^l),f(\sigma^k))}.$$ Moreover, when the orientations on cells of $\Delta(K)$ are induced from the factors, the chain map $\tau_\sharp$ satisfies

$$\tau_\sharp(\sigma\times \tau)=(-1)^{d(\sigma)d({\tau})} (\tau \times \sigma).$$ It follows from the definitions and the above identities that

$$\tau^\sharp \vartheta^m_f(K) = (-1)^m \vartheta^m_f(K).$$ Consequently, $\vartheta^m_f(K)$ is $\tau$-equivariant, or a $\delta$-cochain, if $m$ is even, and is an $s$-cochain if $m$ is odd. Let $\rho_m=\delta$ if $m$ is even and $\rho_m=s$ otherwise.
We state the following facts, for the proofs refer to \cite{Sha57}.

\begin{lemma}
The cochain $\vartheta^m_f(K) \in C^{\rho_m}(\Delta(K))$ is a cocycle.
\end{lemma}

The cocycle $\vartheta^m_f(K)$ is called the \textit{embedding cocycle} of $K$ associated with the map $f$.

\begin{lemma}
The cohomology class $[\vartheta^m_f(K)] \in H^{\rho_m}(\Delta(K))$ is independent of the map $f$.
\end{lemma}

\paragraph{}
The last lemma has an intuitive reasoning, we explain it for the case $m=2d$. Any two maps $K \rightarrow \Rspace^m$ are homotopic. During a PL homotopy, one situation in which 
an intersection point $p$ between $\sigma^d$ and $\tau^d$ is removed is when an intersection point moves, say, to the point $q$ of the boundary of $\sigma^d$, and then disappears from $\sigma^d$. This event is called a \textit{finger move}, and is modeled by placing a small $k$-sphere-with-hole that is linking with the boundary of $\sigma^d$ centered at $q$, next removing a small disk neighborhood of $p$ from $\tau^d$, and, attaching the sphere-with-hole with a narrow tube to the boundary of the removed disk. When the intersection point $q$ disappears, as well as by performing the model finger move, the change in $\vartheta^m(K)$ is $\pm \delta (\sigma^{d-1} \times \tau^d)$, where $\sigma^{d-1}$ is the boundary cell on which $q$ lies.  In general, other events that remove the intersection points do not change the cocycle $\vartheta^m_f$. 

\paragraph{}
Let $\vartheta^m(K)$ denote the cohomology class $[\vartheta^m_f]$. It is called the $m$-dimensional \textit{embedding class} of $K$. The class $\vartheta^{2d}(K)$ is also called \textit{the van Kampen obstruction} for embeddability of $K$ into $\Rspace^{2d}$. Note that the van Kampen obstruction always corresponds to an ordinary cohomology class of the quotient complex $\Delta(K)/\tau$, or a $\delta$-cohomology class of $(\Delta(K), \tau)$.

\paragraph{Remark} Similarly to the Smith classes one can also define the reduced embedding classes as cohomology classes of the quotient complex as follows. If $m$ is even then $$\vartheta^m_r(K)=(\pi^*)^{-1} \vartheta^m(K),$$ and if $m$ is odd, $$ \vartheta^m_r(K)=r_2^*(\pi^*)^{-1} \vartheta^m(K).$$ 

\paragraph{}
Assume $K$ embeddable into $\Rspace^m$. Then clearly $\vartheta^m(K)=0$, and indeed the embedding classes serve as obstructions for the embeddability. Recall that the Smith classes $\Phi^m(K)$ are also obstructions for the embeddability into $\Rspace^m$. One of the main results of \cite{Wu74} is that these two classes coincide, see Chapter V, Section 5 of that book.

\begin{Theorem}
For a (finite) simplicial complex $K$ the embedding classes $\vartheta^m(K)$ and the Smith classes $\Phi^m(K)$ coincide.
\end{Theorem}

It has been proved by Shapiro \cite{Sha57} and Wu \cite{Wu74} that the van Kampen obstruction is complete when $d>2$. As mentioned in the Introduction, when $d=2$, it is known to be incomplete \cite{Freetal94}, and in this case, the embeddability problem is not known to be even decidable. Note that for $d>2$ there is a polynomial algorithm, namely computing $\vartheta^{2d}(K)$ from $\vartheta^{2d}_f(K)$ for an arbitrary $f$. The obstruction is also complete for $d=1$ \cite{Sar91}. We also note that, obviously, the reduced van Kampen obstruction $\vartheta^{2d}_r(K)$ are also complete for $d\neq 2$.

\paragraph{Remark 1}
The embedding class $\vartheta^m(K)$ is an $s$-cohomology class or a $\delta$-cohomology class of the system $(\Delta(K), \tau)$, depending on whether $m$ is odd or even, respectively. Therefore, they do not always correspond to an ordinary cohomology class of $\Delta(K)/\tau$. Since the $d$-cohomology of $(\Delta(K),\tau)$ is isomorphic with the ordinary cohomology of the quotient space $\Delta(K)/\tau$, we could hope to derive all of the Smith class as an ordinary cohomology class of the quotient complex. The $s$-cohomology groups pose a difficulty as was explained in Section \ref{s:smith}. Wu \cite{Wu74} associated an $s$-cohomology class with a class mod 2 of the quotient, which we also followed in our definition of $\vartheta^f_r(K)$. Shapiro \cite{Sha57}, uses ``twisted coefficients" for the cohomology of $\Delta(K)/\tau$. 

\paragraph{Remark 2}
One can give yet another, conceptually easier, definition of the van Kampen obstruction as follows. Consider the infinite dimensional sphere $S^\infty$ and the antipodal action on it. It is easy to define an equivariant map from any cellular $\Zring_2$-space into $S^\infty$, see Section \ref{s:basic}. This map is unique up to $\Zring_2$-equivariant homotopy, see \cite{Hus94}, Theorem 12.4. Hence, there exists a unique homomorphism $\phi_K:H_\delta^*(S^\infty) \rightarrow H^*_\delta(\Delta(K))$. For simplicity, let $m$ be even. Then $H_\delta^m(S^\infty)\cong H^m(\Rspace P^\infty)$. Moreover, there is a unique non-zero element $a_m \in H^m(\Rspace P^\infty)$, see, for instance, \cite{Hat02}, discussion after 3.19 or Section \ref{s:example}. The obstruction class is defined as $\phi_K(a_m)$. Since $A^m(S^\infty)=a_m$, as computed in Section \ref{s:example}, the existence of the map and Proposition \ref{p:equivmap} prove that the class defined here is the same as $A^m(\Delta(K),\tau)$. Since the non-zero element of $H^m(\Rspace P^\infty)$ has order two, it follows that the (reduced) van Kampen obstruction has also order two.

\subsection{Some Properties of the embedding classes}

We have already mentioned that the van Kampen obstruction is an element of order two. 

The reduced and non-reduced embedding classes can be characterized as successive cup products in $H^*(\Delta(K)/\tau,-)$ of the 1-dimensional embedding class with itself, see Shapiro \cite{Sha57} Theorem 4.4, and, Wu \cite{Wu74}, Section II.3 Proposition 11. The coefficients groups, their pairings, and the cup product are different in these two sources. In the latter, cohomology of $\Delta(K)/\tau$ is computed with $\Zring $ coefficients in the even dimensions and with $\Zring_2$ coefficients in the odd dimensions.
While in the former, the special homology groups are treated as a homology of the quotient complex with ``twisted coefficients" and a suitable cup product. These definitions in terms of the cup product can be used to compute a cocycle representative of $\vartheta^m{(K)}, \vartheta_r^m(K)$ for any $m>1$, using a representative for $\vartheta^1(K)$. A representative cocycle for $\vartheta^1(K)$ can also be computed using an ordering of vertices of $K$, see \cite{Sha57}, Theorem 4.3. Then, it is possible even to write down an explicit formula for a cocycle representing any embedding class, using an ordering of the vertices of $K$. This is given explicitly in \cite{Wu74}, Theorem 2 of Section III.2.

\paragraph{Remark}
A polynomial algorithm (assuming $d$ constant) for deciding if $\vartheta^m(K)=0$ can be given as follows. Define an arbitrary general position simplex-wise linear map $f:K \rightarrow \Rspace^m$. Then, compute the cocycle $\vartheta^m_f$ by computing the intersection numbers of disjoint cells of complementary dimensions (This can be made into a combinatorial problem by mapping the vertices to a moment curve, see \cite{Sha57}). Next, calculate the cohomology class of $\vartheta^m_f$ in $H^m_\delta(\Delta(K))$ if $m$ is even, and, in $H^m_s(\Delta(K))$ if $m$ is odd. All these calculations can be done in polynomial time in the number of simplices of $K$. 

\paragraph{}
Let $\Zring_\rho=\Zring$ where it is a coefficient group for an even dimensional cohomology and $\Zring_\rho=\Zring_2$ otherwise.
It is of benefit in many situations to be able to certify that $\vartheta^m(K)$ is not zero. The first guess for such a certificate is a cycle $z \in Z_m(\Delta(K)/\tau, \Zring_\rho)$ such that a representative cocycle for $\vartheta^m_r(K)$ evaluates to a non-zero number on $z$. However, in the case $m$ even one can prove the following.

\begin{lemma}\label{l:zeroooncycle}
Assume $m$ is even and let $\nu \in Z^{m}(\Delta(K)/\tau)$ represent $\vartheta^m_r(K)$. For any cycle $z \in Z_m(\Delta(K)/\tau)$, $\nu(z)=0$.
\end{lemma}

\begin{proof}
Let $\tilde{f}:(\Delta(K),\tau) \rightarrow (S^{\infty},\tau)$ be a $\Zring_2$-equivariant map, and let $f: \Delta(K)/\tau \rightarrow \Rspace P^{\infty}$ be the quotient map. Then, we have, analogous to the equation (\ref{e:functor}), $$f^* (A^{m}_r(\Rspace P^\infty)) = A^{m}_r(\Delta(K)/\tau) = \vartheta^m_r(K),$$ see \cite{Wu74}, Section II.3, Proposition 14.

The map $f$ can be assumed to be cellular, in the sense that, the image of an $i$-dimensional cell of $\Delta(K)/\tau$ is the $i$-dimensional cell of $\Rspace P^\infty$ or the cell is mapped into the lower dimensional cells. Then, for any cycle $z \in Z_m(\Delta(K)/\tau)$, $A^m_r(K)([z]) = f^*A^m_r(\Rspace P^\infty)([z])=A^m_r(\Rspace P^\infty) ([f_\sharp(z)])$. However, since $m$ is even, $Z_m(\Rspace P^\infty)=0$ as a cellular chain group and $[f_\sharp(z)]=0$ since $f_\sharp(z)$ is a cellular cycle.
\end{proof}

\paragraph{Remark} There are well-known situations where the mod 2 reduction of the van Kampen obstruction evaluates to non-zero on a mod 2 cycle. This is the case for instance for the mod 2 reduction of the $4$-dimensional obstruction for $K = \Delta^{(2)}_6$, i.e., the 2-skeleton of the 6-simplex. This is the complex whose non-embeddability into $\Rspace^4$ was proved by van Kampen \cite{vKam33} using the fact just mentioned, and the fact that every other representative of the obstruction must evaluate also to 1 on this cycle. Hence, there can be no embedding of $\Delta^{(2)}_6$ into $\Rspace^4$.

To reconcile the above with Lemma \ref{l:zeroooncycle}, observe that any mod 2 cycle on which the mod 2 reduction of the obstruction evaluates to non-zero must be the mod 2 reduction of chains with non-zero boundary. The case $\Rspace P^\infty$ exemplifies this phenomenon clearly. In \ref{s:example}, we computed that $A^m_r(\Rspace P^\infty)$ is non-zero and is the generator of $H^m(\Rspace P^\infty, \Zring_\rho)$. There is only one cell $\gamma^{m}$ of any dimension and for $m=2k$ we have $\partial \gamma^{2k} = 2\gamma^{2k-1}$. Hence, there is no non-zero cycle in this dimension, and $a_r^{2k}(\gamma^{2k})= \pm 1$, where $A^{2k}_r=[a^{2k}_r]$. Let $r_2$ denote reduction mod 2. Then $r_2(\gamma^{2k})$ becomes a $\Zring_2$-cycle and $r_2^*(A^{2k}(\Rspace P^\infty))([r_2(\gamma^{2k})]))=1$.

\section{The embedding classes of $[3]*K$}\label{s:main}

\subsection{Preliminaries} 
In our arguments, we make use of operations that map cells of $\Delta(K)$ to cells of $\Delta([3]*K)$. We describe these operation here. 

If $v$ is a vertex and $\sigma$ an oriented simplex not containing $v$, we write $v\sigma$ for the oriented simplex $v*\sigma$, where the orientation of $v\sigma$ is defined by adding $v$ as the new first element and then ordering vertices of $\sigma$ as in oriented simplex $\sigma$. With this convention, we have, $$\partial (v \sigma)=\sigma-v(\partial \sigma).$$
Let $v$ be a vertex not in $K$, and let $\sigma_1, \sigma_2 \in K$ be disjoint simplices. Define a function $v:\Delta(K) \rightarrow \Delta(v*K)$ by $$v(\sigma_1 \times \sigma_2) = v\sigma_1 \times \sigma_2.$$
The function $v(\cdot)$ defines a homomorphism $C(\Delta(K)) \rightarrow C(\Delta(v*K))$ that we also denote by $v$.

\begin{lemma}
For any oriented cell $\sigma_1 \times \sigma_2 \in C(\Delta(K))$, $$\partial v(\sigma_1 \times \sigma_2)= \sigma_1 \times \sigma_2 - v(\partial(\sigma_1 \times \sigma_2)).$$

\end{lemma}\label{l:partialv}

\begin{proof}
Using the formula for the boundary of a product cell, $$\partial(\sigma_1 \times \sigma_2)=\partial \sigma_1 \times \sigma_2 + (-1)^{d(\sigma_1)}\sigma_1 \times \partial \sigma_2,$$
we compute, 
\begin{equation*}
\begin{split}
\partial v(\sigma_1 \times \sigma_2)&= \sigma_1 \times \sigma_2 - v\partial\sigma_1\times\sigma_2 + (-1)^{d_1+1}v\sigma_1\times\partial\sigma_2\\
&= \sigma_1 \times \sigma_2 - v(\partial\sigma_1\times \sigma_2)+(-1)^{d_1+1} v(\sigma_1\times \partial\sigma_2)\\
&= \sigma_1\times \sigma_2 - v(\partial(\sigma_1 \times \sigma_2)). 
\end{split}
\end{equation*}
\end{proof}

Let $v,w$ be two vertices not in $K$, and let $\sigma_1, \sigma_2 \in K$ be disjoint simplices. Define the function $vw:\Delta(K) \rightarrow \Delta(\{v,w\}*K)$ by
$$ vw(\sigma_1 \times \sigma_2) = (-1)^{d(\sigma_1)}v\sigma_1 \times w\sigma_2.$$

We write also $vw$ for the homomorphism $C(\Delta(K)) \rightarrow C(\Delta(\{v,w\}*K))$ defined by the function $vw$.
\begin{lemma}\label{l:partialvw}
For any oriented cell $\sigma_1 \times \sigma_2 \in C(\Delta(K))$,$$\partial(vw(\sigma_1\times\sigma_2))=(-1)^{d_1}\sigma_1 \times w \sigma_2 -v\sigma_1 \times \sigma_2+vw(\partial(\sigma_1 \times \sigma_2)). $$
\end{lemma}
\begin{proof}
We just compute 
\begin{equation*}
\begin{split}
\partial(vw(\sigma_1 \times \sigma_2))&=(-1)^{d_1} \sigma_1\times w\sigma_2+(-1)^{d_1+1}v\partial\sigma_1\times w \sigma_2\\&+(-1)^{2d_1+1}v \sigma_1 \times \sigma_2  +(-1)^{2d_1+2} v\sigma_1 \times w \partial \sigma_2\\
&= (-1)^{d_1}\sigma_1 \times w \sigma_2+ (-1)^{1}v \sigma_1\times \sigma_2+vw(\partial \sigma_1\times \sigma_2)\\&+(-1)^{d_1}vw(\sigma_1 \times \partial \sigma_2)\\
&=(-1)^{d_1}\sigma_1 \times w \sigma_2 -v\sigma_1 \times \sigma_2+vw(\partial(\sigma_1 \times \sigma_2)).
\end{split}
\end{equation*}
\end{proof}

\begin{lemma}\label{l:tausharp}
For an oriented cell $\sigma_1 \times \sigma_2 \in C(\Delta(K))$ we have $$\sigma_1 \times w \sigma_2 = (-1)^{d_1}\tau_\sharp w(\tau_\sharp(\sigma_1 \times \sigma_2))$$
\end{lemma}

\begin{proof}
We have
\begin{equation*}
\begin{split}
\tau_\sharp w (\tau_\sharp(\sigma_1 \times \sigma_2))&= (-1)^{d_1d_2} \tau_\sharp w (\sigma_2 \times \sigma_1)\\
&= (-1)^{d_1d_2} \tau_\sharp(w\sigma_2 \times \sigma_1)\\
&= (-1)^{d_1d_2+(d_2+1)d_1}\sigma_1 \times w \sigma_2\\
&= (-1)^{d_1} \sigma_1 \times w \sigma_2.
\end{split}
\end{equation*}
\end{proof}

From lemmas \ref{l:tausharp} and \ref{l:partialvw} follows that
\begin{equation}\label{e:boundary}
\begin{split}
\partial(vw(\sigma_1 \times \sigma_2))=  \tau_\sharp(w (\tau_\sharp(\sigma_1 \times \sigma_2))) -v(\sigma_1 \times \sigma_2)+vw(\partial(\sigma_1 \times \sigma_2)).
\end{split}
\end{equation}

\subsection{Complexes with non-zero mod 2 obstruction classes}
Our proof is simpler if we work with mod 2 homology and obstruction classes. Therefore, we first present the simpler case where we assume that the mod 2 obstruction class is non-zero and later we show the general case. We prove the following theorem.

\begin{Theorem}\label{t:mod2}
Let $K$ be a simplicial $d$-complex. Assume $r_2^*(\vartheta^{m}_r(K)) \neq 0$, where $r_2^*: H^{m}(\Delta/\tau) \rightarrow H^{m}(\Delta/\tau,\Zring_2)$ is induced by the reduction mod 2 of cochains. Then $r_2^*(\vartheta^{m+2}_r([3]*K)) \neq 0$.
\end{Theorem}

\begin{proof}
For any $m$ let $A_r^m=\vartheta_r^m = \vartheta_r^m(K)$. From the assumption it follows that there exists a mod 2 cycle $\tilde{z} \in Z_{m}(\Delta(K)/\tau, \Zring_2)$ such that $r_2^*(A^{m}_r)([\tilde{z}]) \neq 0$.

Let $\{ \Gamma_i\}$ be a set of $m$-cells of $\Delta(K)$ forming a fundamental domain for all the $m$-cells. Then, $\{[\Gamma_i]\}$ is the set of $m$-cells of $\Delta(K)/\tau$. Therefore, $\tilde{z}$ might be written as $$\tilde{z} = \sum_{j \in J} [\Gamma_j]$$ for some index set $J$. Now let $z \in Z_{m}^\delta(\Delta(K),\Zring_2)$ be defined as $$z = \sum_{j \in J} \Gamma_j + \tau_\sharp \Gamma_j=(1+\tau_\sharp)\sum_{j\in J}\Gamma_j = \bar{\pi}_\sharp(\tilde{z}).$$
Note that since $\bar{\pi}_\sharp$ is a chain map (\cite{Wu74} II.1 Proposition 2), $z$ is a $\delta$-cycle. Let $y=\sum_{j\in J} \Gamma_j$, then, $z=(1+\tau_\sharp)y$, and, $\pi_\sharp(y)=\tilde{z}.$

Let the three vertices of $[3]$ be $u,v$ and $w$. We assume that the functions $u,v,w,uv, \ldots$ defined in the previous section map into $[3]*K$. Consider the chain $\zeta \in C_{m+2}(\Delta([3]*K),\Zring_2)$ defined as
$$\zeta = vw(z)+wu(z)+uv(z).$$

A calculation using equation (\ref{e:boundary}) shows that
$$\partial(\zeta) = v(z)+u(z)+w(z) + \tau_\sharp(w(\tau_\sharp(z)) + \tau_\sharp(v(\tau_\sharp(z))) + \tau_\sharp(u(\tau_\sharp(z))).$$

Since $z$ is a mod 2 $\delta$-cycle $(1-\tau_\sharp)z=0$, hence $\tau_\sharp z = z$. The above can be written
$$\partial(\zeta) = (1+\tau_\sharp)(v(z) + u(z) + w(z)).$$
It follows that $\pi_\sharp (\zeta)$ is a $\Zring_2$-cycle in $\Delta([3]*K)/\tau$.

We work here with the special and reduced Smith classes that are defined using $\Zring_2$ coefficients instead of $\Zring$. Let $A^n_2(-,\tau) \in H^n(-, \Zring_2)$ denote the special classes. We have $r_2^*\mathbf{1} = \mathbf{1}_2$, i.e., $r_2^*A^0 = A^0_2.$ From \cite{Wu74}, II.3, Proposition 3, the reduction mod 2, $r_2$, commutes with the maps $\mu^i_\delta, \mu^i_s$. Therefore, $r_2^*A^1 = r_2^*\mu_\delta^1(\mathbf{1})=\mu_\delta^1(r_2^*A^0)=A^1_2$, etc. It follows that $$ r_2^*A^n(K,\tau) = A^n_2.$$
Therefore it is enough to show that $A^{m+2}_2([3]*K) \neq 0$ which we do in the following.

Let $A_2 = A_2^{m}([3]*K)$ so that $\mu_\delta^{m,2}(A_2)=A_2^{m+2}([3]*K).$ Let moreover $\iota: \Delta(K) \rightarrow \Delta(K\times[3])$ be the inclusion. Then $A_2^{m}(K)=\iota^* A_2$. 

Let $\psi_0,\psi_1,\ldots$ be a mod 2 resolution for the (mod 2) cocycle $1 \in C^0(\Delta([3]*K)/\tau, \Zring_2). $ Then, $\delta \psi_i = (1+\tau_\sharp) \psi_{i+1}$ and recall that $A_{2,r}^{i}([3]*K)=[\bar{\pi}^\sharp \psi_i]$.

We calculate
\begin{equation*}
\begin{split}
    A_{2,r}^{m+2}([\pi_\sharp \zeta]) &= (\bar{\pi}^\sharp \psi_{m+2})(\pi_\sharp \zeta)\\
    &= {({\pi}^\sharp)}^{-1}(\delta\psi_{m+1}) (\pi_\sharp \zeta)\\
    &=\delta\psi_{m+1}(\zeta) = \psi_{m+1}(\partial \zeta) \\&= \psi_{m+1}((1+\tau_\sharp)(v(z)+u(z)+w(z)))\\
    &=(1+\tau^\sharp)(\psi_{m+1})(v(z)+u(z)+w(z)) \\&= \delta \psi_{m} (v(z)+u(z)+w(z))\\
    &= \psi_{m}(z) = (1+\tau^\sharp)\psi_{m}(y)\\
    &={\pi^\sharp} (\pi^\sharp)^{-1} (1+\tau^\sharp) \psi_{m} (y)\\
    &= {(\pi^\sharp)}^{-1} (1+\tau^\sharp) \psi_{m} (\pi_\sharp y) = {\bar{\pi}^\sharp}\psi_{m} (\pi_\sharp y).
    \end{split}
\end{equation*}
Now observe that $\tilde{z}=\pi_\sharp y$ is a chain in $\Delta(K)/\tau \subset \Delta([3]*K)/\tau$ and $\iota^\sharp$ is the restriction of a cochain, therefore, $\bar{\pi}^\sharp \psi_{m} (\tilde{z})= \iota^\sharp \bar{\pi}^\sharp \psi_{m} (\tilde{z})=A_{2,r}^{m}(K)([\tilde{z}])\neq 0$. It follows that $A_{2,r}^{m+2}([\pi_\sharp \zeta]) \neq 0$, it follows that $A^{m+2}_{2,r}([3]*K) \neq 0$.
\end{proof}

\subsection{Complexes with non-zero integer obstruction classes}

We start with an algebraic lemma that provides us with a method to certify that a cocycle represents a non-zero torsion cohomology class by evaluating it on a chain. We have not encountered such a lemma in the literature, however, this is a simple result of homological algebra. 

\begin{lemma}[Certificate for Elements of Ext]\label{l:generalchain}
Let $L$ be a general cell complex. Let $\phi$ be an (integer) $i$-cocycle and and $c$ an (integer) $i$-chain, such that, $\partial c=n d$ for an integer $n>1$ and an $(i-1)$-cycle $d$. Assume $\phi(c)\neq 0$ mod $n$. Then $[\phi] \neq 0$ as an integer cohomology class. Conversely, if $\phi(z)=0$ for all $i$-cycles \footnote{Note that in this case $[\phi]\in \text{Ext}(H_{i-1}(K), \Zring)$, see the Universal Coefficients Theorem} $z$ and $[\phi]\neq 0$, then there exist $c,d$ and an $n>1$ such that $\partial c= nd$ and $\phi(c) \neq 0$ mod $n$.
\end{lemma}

\begin{proof}
Assume $[\phi]=0$, then, there is an $(i-1)$-cochain $\psi$ such that $\phi = \delta \psi$. We have $\phi(c)=\delta \psi(c)=\psi(\partial c)=\psi(nd)=n \psi(d)$. But this is in contradiction with our assumption that $\phi(c) \neq 0$ mod $n$.

To prove the other direction, assume $\phi$ vanishes on all cycles but $[\phi]\neq0$. We consider the short exact sequence of chain complexes

\begin{equation*}
\begin{split}
0 \rightarrow Z \rightarrow C \xrightarrow{\partial} B \rightarrow 0.
\end{split}
\end{equation*}
In the above, the cycle groups $Z_n(L)$ are made into a chain complex $Z$ with differential 0. Similarly for the boundary groups $B_n$. The maps $Z_n \rightarrow C_n$ are the inclusions. Applying $\text{Hom}(, \Zring)$ to this sequence we obtain a dual short exact sequence
\begin{equation*}
\begin{split}
0 \rightarrow B^* \xrightarrow{\delta} C^* \xrightarrow{r} Z^* \rightarrow 0.
\end{split}
\end{equation*}
The homomorphism $r: C_n^* \rightarrow Z_n^*$ is the dual to inclusion which is the same as restriction of homomorphisms to $Z_n$. The long exact sequence of homology groups associated with this short exact sequence of chain complexes is 

\begin{equation}
\begin{split}
\cdots Z^*_{n-1} \rightarrow B_{n-1}^* \xrightarrow{\delta} H^n(C) \xrightarrow{r} Z^*_n \rightarrow \cdots.
\end{split}
\end{equation}

The boundary map of the long exact sequence $Z^*_{n-1} \rightarrow B^*_{n-1}$ can be seen to be the same as restriction to boundaries. See \cite{Hat02} page 192 for full details of the above sequence. 
The map $B^*_{n-1} \rightarrow H^n(C)$ is induced by $\delta$, and the map $H^n(C) \rightarrow Z^*_n$ is induced by the restriction of cocycles to $Z_n$. Now since $\phi$ vanishes on $Z_n$ we must have $r([\phi])=0.$ Since the sequence is exact and $[\phi]\neq 0$, there is a homomorphism $\psi \in B^*_{n-1}$ such that $[\delta \psi]=[\phi]$. Therefore $\psi$ is not in the kernel of $\delta$, which is the same as the image of $Z^*_{n-1}$. Consequently, $\psi$ is a homomorphism of $B_{n-1}$ that cannot be extended to $Z_{n-1}$, and, $\phi = \delta \psi +\delta \theta$ where $\theta$ is defined on all of $C_{i-1}$.

Let us write $H_{i-1}(L) = \bigoplus_{j=1}^m \Zring/n_j\Zring \oplus \bigoplus_{j=m+1}^{s} \Zring$ and assume that $z_{1}, \ldots,z_m$ are $(i-1)$-cycles representing generators of $\Zring/n_j\Zring$ factors. Moreover, let $c_j$ be such that $\partial c_j = n_j z_j$. For simplicity, we set $n_j=\infty$ for $j>m$. Let $Z'$ be the (free) subspace of $Z_{i-1}(L)$ generated by the $z_i$ and the generators of $\Zring$ factors of $H_{i-1}(L)$. If we could extend $\psi$ to $Z_{i-1}$ it would follow that $[\phi]=0$, contradicting our assumption.

The homomorphism $\psi$ is already defined on all the cycles that are boundaries. On the other hand, if for all $j$, $\psi(n_j z_j)= 0$ mod $n_j$, then, $\psi$ can be defined on $z_j$ consistently with its values on boundaries. In addition, define $\psi$ on generators of $\Zring$ factors of $Z'$ arbitrarily. Then we have defined $\psi$ on all of $Z'$. For any other cycle $z \in Z_{i-1}(L)$, that is not a boundary, there is a unique $z'\in Z'$, of the form $\Sigma m_j z_j$ with $m_j < n_j$, such that $z-z'=b \in B$. Hence, such $b$ is also unique. We define $\psi(z)= \psi(z')+\psi(b)$. This defines $\psi$ an all the cycles. It is easily checked that $\psi$ thus defined is a homomorphism and we obtain the desired contradiction. Since this is not possible there is some $j$ for which $\phi(c_j) =\phi(c_j)-n_j \theta(z_j)= \psi(n_j z_j)\neq 0$ mod $n_j$.

\end{proof}

We are ready for the general case of our main result.
\begin{Theorem}
Let $K$ be a simplicial $d$-complex such that $\vartheta^{m}(K) \neq 0$. Then $\vartheta^{m+2}([3]*K)\neq 0$, $m$ even.
\end{Theorem}

\begin{proof}
Let $\psi_0, \psi_1, \psi_2, \ldots$ be a resolution for the cocycle $1\in C^0(\Delta([3]*K)/\tau)$ so that $[\bar{\pi}^\sharp \psi_{j}]=A^{j}_r([3]*K)$. Let $\iota: \Delta(K)\rightarrow \Delta([3]*K)$ be the inclusion. Then, from Lemma~\ref{l:resolution} below follows that the sequence $\iota^\sharp \psi_k$ defines a resolution for $1\in C^0(\Delta(K)/\tau)$. From the assumption that $\vartheta^{m}(K) \neq 0$ and
Lemma \ref{l:zeroooncycle}, Lemma \ref{l:generalchain}  say that there is an integer $n$ and a chain $\tilde{c} \in C_{m}(\Delta(K)/\tau)$ such that $\bar{\pi}^\sharp\iota^\sharp \psi_{m}(\tilde{c}) \neq 0$ mod $n$, and, $\partial \tilde{c} = n \tilde{z}$ for a cycle $\tilde{z}$. The chain $\tilde{c}$ corresponds to a $\delta$-chain $c=\bar{\pi}_\sharp(\tilde{c}) \in C_{m}(\Delta(K)).$ Let $c = (1 + \tau_\sharp) y$. We have, $\partial c=\partial \bar{\pi}_\sharp (\tilde{c})=\bar{\pi}_\sharp \partial \tilde{c} =n \bar{\pi}_\sharp \tilde{z}$. Let $z=\bar{\pi}_\sharp \tilde{z}$. 

We construct the chain $\zeta \in C_{m+2}(\Delta([3]*K))$ by the formula $$\zeta = vw(c)+wu(c)-vu(c).$$ 
We first compute,
\begin{equation*}
\begin{split}
\partial(\zeta)&=\tau_\sharp(w (\tau_\sharp(c))) - v(c)+vw(\partial(c))\\
&+ \tau_\sharp(u (\tau_\sharp(c))) - w(c)+wu(\partial(c))\\
&- \tau_\sharp(u (\tau_\sharp(c))) + v(c)-vu(\partial(c)).
\end{split}
\end{equation*} Now since $\tau_\sharp(c)=c$ and $\partial c=n z$ the above simplifies to
\begin{equation*}
\begin{split}
\partial(\zeta)& = (\tau_\sharp-1) w(c)+ \big(n(vw(z)+wu(z)-vu(z))\big).
\end{split}
\end{equation*}
It follows that $\partial \pi_\sharp(\zeta)=\pi_\sharp \partial(\zeta) = n \pi_\sharp(vw(z)+wu(z)-vu(z))$. We set $d=\pi_\sharp(vw(z)+wu(z)-vu(z)).$ Thus, $\partial \pi_\sharp(\zeta)=nd$.

We have 
\begin{equation*}
\begin{split}
\bar{\pi}^\sharp \psi_{m+2}(\pi_\sharp \zeta)&= (\pi^\sharp)^{-1}(1+\tau^\sharp)(\psi_{m+2})(\pi_\sharp \zeta)\\
&=(\pi^\sharp)^{-1} \delta \psi_{m+1}(\pi_\sharp \zeta)\\
&=\delta \psi_{m+1}(\zeta)=\psi_{m+1}(\partial \zeta)\\
&=-\psi_{m+1}((1-\tau_\sharp)w(c))+n\psi_{m+1}(vw(z)+wu(z)-vu(z))\\
&=-\psi_{m}(\partial(w(c)) \; \text{mod} \; n\\
&=-\psi_{m}(c-w(\partial c)) = -\psi_{m}(c)+n \psi_{m}(w(z)) = -\psi_{m}(c)  \; \text{mod} \; n.\\
\end{split}
\end{equation*}

Now since $c$ is a chain of $\Delta(K)$, $\psi_{m}(c)=\iota^\sharp \psi_{m}(c)$. We can write $\iota^\sharp\psi_{m}(c) = \pi^\sharp (\pi^\sharp)^{-1}\iota^\sharp \psi_{m}(1+\tau_\sharp y) = \pi^\sharp {\bar{\pi}}^\sharp \iota^\sharp \psi_{m} (y)=\bar{\pi}^\sharp\iota^\sharp \psi_{m}(\pi_\sharp y) = \bar{\pi}^\sharp\iota^\sharp \psi_{m}(\tilde{c}).$ However, we have by the choice of $\tilde{c}$, $\bar {\pi}^\sharp \iota^\sharp \psi_{m}(\tilde{c}) \neq 0$ mod $n$. It follows that $\bar{\pi}^\sharp \psi_{m+2}(\pi_\sharp \zeta) \neq 0$ mod $n$. Since $\partial \pi_\sharp \zeta = n d$, lemma \ref{l:generalchain} shows that $\vartheta^{m+2}_r([3]*K) = A^{m+2}_r([3]*K)=[\bar{\pi}^\sharp \psi_{m+2}] \neq 0,$ from which the theorem follows.
\end{proof}

\begin{lemma}\label{l:resolution}
	Let $\psi_0, \psi_1, \psi_2, \ldots $ be a resolution for the cocycle $1 \in C^0(\Delta([3]*K)/\tau)$, and, let $\iota: \Delta(K) \subset \Delta([3]*K)$ be the inclusion. Then, $\iota^\sharp \psi_0, \iota^\sharp \psi_1, \ldots$ is a resolution for $1 \in C(\Delta(K)/\tau)$.
\end{lemma}

\begin{proof}
	Observe that $\bar{\pi}^\sharp\iota^\sharp \psi_0$ assigns 1 to each vertex hence $\bar{\pi}^\sharp\iota^\sharp \psi_0 = 1 \in C^0(\Delta(K)/\tau)$. Moreover, from the fact that $\iota_\sharp$ is an equivariant chain map it follows that $$\delta \iota^\sharp \psi_k= \iota^\sharp \delta \psi_k = \iota^\sharp (1+(-1)^{k+1} \tau^\sharp) \psi_{k+1}=(1+(-1)^{k+1}\tau^\sharp)\iota^\sharp \psi_{k+1}.$$ Therefore, $\iota^\sharp \psi_k$ define a resolution for $1 \in C^0(\Delta(K)/\tau)$.
\end{proof}

\subsection{Embedding $[3]*K$ into $\Rspace^{2d+2}$ when $\vartheta(K)=0$}\label{s:embedding}
In this section we prove that, if $K$ is a $d$-dimensional simplicial complex with $\vartheta^{2d}(K)=0$, then, the $(d+1)$-complex $[3]*K$ has vanishing van kampen obstruction, i.e., $\vartheta^{2d+2}([3]*K)=0$. Note that, for $d\neq 2$, $\vartheta^{2d}(K)=0$ implies $K$ embeddable into $\Rspace^{2d}$, from which follows easily that $[3]*K$ embeds into $\Rspace^{2d+2}$. Therefore, the argument given in this section is needed only for the case $d=2$. This argument finishes the proof of Theorem \ref{t:main}. The method of proof is by presenting a PL map with zero embedding cocycle. This seems to be easier than proving $\vartheta^{2d+2}([3]*K)=0$ algebraically from $\vartheta(K)=0$.

Let $d\geq 2$. We start by taking a (PL) map $f:K \rightarrow \Rspace^{2d}$ with $\vartheta^{2d}_f(K)=0$. This map exists by the standard theory of embeddings. Namely, one defines an arbitrary PL map in general position and exploits finger moves to arrive at a map with vanishing intersection cocycle. Consequently, the images of every two disjoint $d$-simplices under $f$ have intersection number zero in $\Rspace^{2d}$. 

Now we think of $\Rspace^{2d}$ as lying in some $\Rspace^{2d+2}$. There exists a family of $(2d+1)$-spaces, or hyperplanes if $\Rspace^{2d+2}$ whose mutual intersection is exactly this $\Rspace^{2d}$, and, together they cover $\Rspace^{2d+2}$. These $(2d+1)$-spaces can be parametrized by points of an $S^1$. Let $p \in \Rspace^{2d} \subset \Rspace^{2d+2}$, and let $B(p)$ be a closed $(2d+2)$-ball centered at $p$ with sufficiently small radius. We schematically describe the intersection of these $(2d+1)$-hyperplanes with the ball $B(p)$ as diameters of the $S^1$, as in Figure \ref{f:embedding}.

Take three points of $S^1$, say $\alpha_1, \alpha_2, \alpha_3$. For $\alpha_i$ we define a vertex $v_{\alpha_i}$. Put $v_{\alpha_i}$ in the $\Rspace^{2d+1}$ space with parameter $\alpha_i$, and not in our $\Rspace^{2d}$. Next, cone the image of $f$ using these three vertices. The cones lie in their respective $(2d+1)$-spaces. This gives us a map of $[3]*K$ into $\Rspace^{2d+2}$, which we denote also by $f$. Note that the $v_{\alpha_i}$ will be the image of the vertex $i$ in our map of $[3]*K$. 

We next modify $f$ in two steps. The first step is a local modification in a ball of small radius around each intersection point of two disjoint $d$-simplices of $K$. Let $p$ be an intersection point of (images of) two disjoint $d$-simplices $\sigma_1,\sigma_2$ of $K$ in $\Rspace^{2d} \subset \Rspace^{2d+2}$. Take the ball $B(p)$ with a radius such that the map $f$ is simplex-wise linear inside (preimage of) $B(p)$. The intersection of $B(p)$ with image of $f$ then consists of six $(d+1)$-dimensional half-balls. They are divided into two sets of three each, say $\Sigma_1, \Sigma_2$. For $i=1,2$, the elements of $\Sigma_i$ are part of cones over $f(\sigma_i)$ from the three $v_{\alpha_j}$. Elements of each $\Sigma_i$ coincide on their boundary in $f(\sigma_i)$. Therefore, an element of $\Sigma_1$ intersects an element of $\Sigma_2$ in a line segment or in a point. If the two elements belong to the same cone, they intersect at a line segment ending at $p$. If they belong to different cones, they intersect at the point $p$ only, which lies in their boundary.

\begin{figure}
	\centering
	
	\includegraphics[scale=0.4]{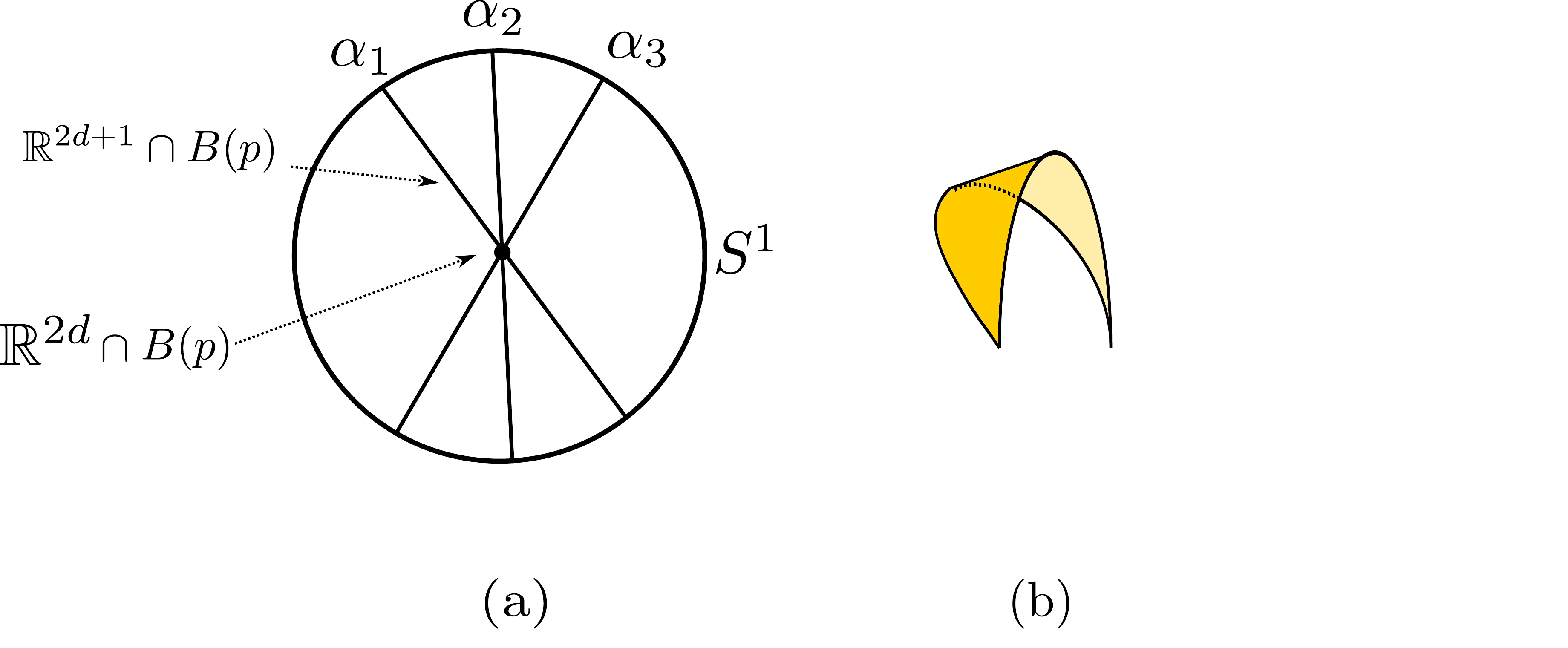}
	\caption{(a) Schematic view of the intersection of spaces with $B(p)$, (b) An angular strip when $d=1$}\label{f:embedding}
\end{figure}

We modify our map as follows. Let $A_i \in \Sigma_1$ be the half-ball inside the space $\alpha_i$. For any $i$ let $a_i$ be the $d$-disk $cl(\partial A_i-\Rspace^{2d})$ lying on the sphere $\partial B(p)$. Remove the interior of $A_2$. We fill $\partial A_2$ again by attaching an ``angular strip" with boundary $a_2 \cup a_3$ and then using $A_3$ to fill the hole. This angular strip lies in the $(2d+1)$-spaces parametrized by the arc $\alpha_2\alpha_3$ of $S^1$, see Figure \ref{f:embedding}. This operation changes the map $f$ on a small neighborhood of the preimage of $p$ in $v_{2}*\sigma_1$.

We do similarly for $A_1$ by removing its interior and connecting its boundary by an angular strip to $\partial A_2$. Then, we fill the hole using the angular strip for $A_2$ and $A_3$. It follows that the images of all three cones over $\sigma_1$ coincide in the interior of the ball $B(p)$. Moreover, we have introduced no new intersection between disjoint $(d+1)$-simplices outside the small balls. 

Perform the same operation on the half-balls of $\Sigma_2$. Let $\hat{f}:[3]*K \rightarrow \Rspace^{2d+2}$ be the resulting map after modifying $f$ as above near each intersection point of two disjoint $d$-simplices. Any intersection point $p$ of two $d$-simplices is now an endpoint of an intersection arc of $3\times 2$ ordered pairs of disjoint $(d+1)$-simplices.

The second modification on $\hat{f}$ goes as follows. Take again a very small ball $B(p)$ centered at the intersection of two $d$-simplices of $K$. Now inside the ball $B(p)$ we see two $(d+1)$-half-disks which intersect at an arc. It is easy to see that we can perturb the two half-disks inside $B(p)$ such that this intersection arc becomes a transversal intersection point of the two half-disks, consequently of the corresponding $(d+1)$-simplices. Do this for all $p$. Let the the resulting map be $\tilde{f}$.

From the above construction, one can see that, for $i\neq j$, two simplices $v_{\alpha_i}*\sigma_1, v_{\alpha_j}*\sigma_2$  are disjoint and intersect transversally at a point in $B(p)$ under $\tilde{f}$ if and only if $\sigma_1, \sigma_2$ are disjoint and intersect at $p$ under $f$. Once can easily check that the signs of the intersections are defined naturally by those of intersections of $\sigma_1, \sigma_2$. It follows that the map $\tilde{f}$ satisfies $\vartheta^{2d+2}_{\tilde{f}}=0$. Moreover, since $d\geq 2$, this latter condition is enough for existence of an embedding into $\Rspace^{2d+2}$. This finished the proof of both directions of Theorem~\ref{t:main}.

\bibliographystyle{spmpsci}      

\end{document}